\documentclass[12pt]{article}
\usepackage{amssymb,amsthm,amsmath,pb-diagram,color,enumerate}
\usepackage{mathrsfs, mathabx}
\usepackage[all]{xy}
\usepackage{url}
\usepackage{fullpage}
\usepackage[OT2,T1]{fontenc}
\DeclareSymbolFont{cyrletters}{OT2}{wncyr}{m}{n}
\DeclareMathSymbol{\Sha}{\mathalpha}{cyrletters}{"58}

\newtheorem{thm}{Theorem}
\newtheorem{lemma}[thm]{Lemma}

\newtheorem{prop}[thm]{Proposition}
\newtheorem{cor}[thm]{Corollary}
\theoremstyle{definition}
\newtheorem{remark}[thm]{Remark}

\newtheorem{notation}[thm]{Notation}

\newtheorem{defn}[thm]{Definition}
\newtheorem{example}[thm]{Example}

\newcommand{\mbb}[1]{\mathbb #1}

\newcommand{\mc}[1]{\mathcal #1}
\newcommand{\ms}[1]{\mathscr #1}

\newcommand{\oper}[1]{\operatorname{#1}}
\newcommand{\wh}{\widehat}
\newcommand{\til}{\widetilde}

\newcommand{\Spec}{\oper{Spec}}
\newcommand{\Ind}{\oper{Ind}}
\newcommand{\Gal}{\oper{Gal}}

\newcommand{\Hom}{\oper{Hom}}
\newcommand{\End}{\oper{End}}
\newcommand{\Aut}{\oper{Aut}}

\newcommand{\sep}{\mathrm{sep}}

\newcommand{\cha}{\oper{char}}

\newcommand{\Frac}{\oper{frac}}

\newcommand{\h}{^{\oper{h}}}
\newcommand{\red}{^{\oper{red}}}

\newcommand{\Obj}{\oper{Obj}}
\def\<{\left<}
\def\>{\right>}

\newcounter{itemcounter}

\title{Local-global Galois theory of arithmetic function fields}
\author{David Harbater, Julia Hartmann, Daniel Krashen,\\ Raman Parimala, Venapally Suresh}
\date{}
\numberwithin{thm}{section}
\begin{document}
\maketitle

\begin{abstract}\let\thefootnote\relax
We study the relationship between the local and global Galois theory of function fields over a complete discretely valued field.  We give necessary and sufficient conditions for local separable extensions to descend to global extensions, and for the local absolute Galois group to inject into the global absolute Galois group.  As an application we obtain a local-global principle for the index of a variety over such a function field.  In this context we also study algebraic versions of van Kampen's theorem, describing the global absolute Galois group as a direct limit of local absolute Galois groups.
\end{abstract}
 
\section{Introduction} \label{intro sec}

In this paper we relate the local and the global Galois theory of function fields $F$ of curves over a complete discretely valued field $K$.  Each such curve has a normal projective model $\mc X$ over the valuation ring $T$ of $K$.  Given a closed point $P$ on $\mc X$, one can compare the Galois theory of $F$ to that of the fraction field $F_P$ of the complete local ring $\wh R_P$ of $\mc X$ at $P$.  In particular, is every finite separable extension of $F_P$ induced by a finite separable extension of $F$?  
That is, for every finite separable field extension $E_P/F_P$, is there a finite separable field extension $E/F$ such that $E \otimes_F F_P$ is isomorphic to $E_P$?
As a related question, is the homomorphism of absolute Galois groups $\Gal(F_P) \to \Gal(F)$ an inclusion?  The answers turn out to depend on the situation.  

We show that the answer to the first question is yes if and only if the residue field $k$ of $T$ has characteristic zero and the closed fiber $X$ of $\mc X$ is unibranched at $P$; and that the answer to the second question is yes if and only if $\cha(k)=0$.  (See the paragraph preceding Theorem~\ref{point to global} for the terminology.)
In \cite{CHHKPS 0cyc}, we considered a related question: in the situation as above, let $U$ be a nonempty connected affine open subset of $X$, and let $F_U$ be the fraction field of the ring $\wh R_U$ of formal functions along $U$.  (See Notation~\ref{basic notation}.)  Then is every finite separable extension of $F_U$ induced by a finite separable extension of $F$?  There we showed that the answer to that question is always yes, regardless of characteristic.  In the current paper, we use that to show that the homomorphism $\Gal(F_U) \to \Gal(F)$ is always an inclusion.  

These results raise the question of how $\Gal(F)$ is related to the groups $\Gal(F_P)$ and $\Gal(F_U)$, if we pick a finite set $\mc P$ of closed points $P$ and let $\mc U$ be the set of connected components of the complement of $\mc P$ in $X$, such that each element of $\mc U$ is affine.  
An important situation is if the reduction graph of $\mc X$ is a tree (see the sentence before Theorem~\ref{dir lim patches thm} for the definition).  Under this hypothesis, we show that $\Gal(F)$ is a direct limit of local Galois groups.  In particular, in the special case of one unibranched point $P$ and its complement $U$, this gives an analog of van Kampen's theorem in topology.  More generally, without the tree hypothesis, we obtain a description of $\Gal(F)$ in terms of groupoids, as well as a description (via a result of J.~Stix) in terms of a maximal subtree of the reduction graph.

As an application of our descent results, we obtain a more explicit version of a local-global principle that appeared in \cite{CHHKPS 0cyc}.  That result concerned the index of a variety over $F$, and related it to its index over the fields $F_P$ and $F_U$.  The version that we prove here is in the equal characteristic zero case, and it relies on the descent results mentioned above.  

The structure of the paper is as follows:  Section~\ref{descent sec} concerns the question of descent of finite separable extensions from $F_P$ to $F$.  In Section~\ref{descent char 0} we provide a positive answer if $\cha(k)=0$ and $X$ is unibranched at $P$ (Theorem~\ref{point to global}), but show that there are always counterexamples if $X$ is not unibranched at $P$ (Remark~\ref{char p descent rks}(\ref{unibranch rk})).  In Section~\ref{lgp applic sec} we combine Theorem~\ref{point to global} and a result from \cite{CHHKPS 0cyc} to obtain an explicit local-global principle for zero-cycles on varietes over $F$ under the characteristic zero hypothesis (Corollary~\ref{prod bound}).  We show in Section~\ref{char p subsec} that if $\cha(k) = p > 0$, then there are always degree $p$ separable extensions of $F_P$ that do not descend to extensions of $F$ (Proposition~\ref{counterexample prop}). 

Section~\ref{Gal gps sec} concerns the implications for absolute Galois groups.  In Section~\ref{inj subsec} we show that $\Gal(F_P) \to \Gal(F)$ is injective if and only if $\cha(k)=0$, and that $\Gal(F_U) \to \Gal(F)$ is always injective (Theorem~\ref{inj Gal maps}).  Section~\ref{van Kampen subsec} obtains a van Kampen theorem in a simple case (Theorem~\ref{diamond vK}, for diamonds),
with generalizations given in Section~\ref{genl van Kampen subsec} (Theorem~\ref{dir lim patches thm}, for trees; and Theorem~\ref{general vK}, in terms of groupoids).

We thank Jean-Louis Colliot-Th\'el\`ene, Florian Pop, and Jakob Stix for helpful discussions.  We thank the American Institute of Mathematics for helping to facilitate this project with an AIM SQuaRE. We recently learned that G\"otz Wiesend once gave a talk that was related to questions that we are now considering in this paper.  We wish that we had had the opportunity to discuss this with him.

\section{Descent of extensions from local to global} \label{descent sec}

\subsection{Descent in characteristic zero} \label{descent char 0}

As in \cite{HH:FP}, \cite{HHK}, and \cite{HHK:refinements}, we consider the following situation:

\begin{defn}
Let $K$ be a complete discretely valued field with valuation ring $T$, uniformizer~$t$, and residue field $k$.  A {\em semi-global field} is a one-variable function field over such a field $K$.
A {\em normal model} of a semi-global field $F$ is a $T$-scheme $\mc X$ with function field $F$ that is flat and projective over $T$ of relative dimension one, and that is normal as a variety.  The {\em closed fiber} of ${\mc X}$ is $X:={\mc X}_k$.
\end{defn}

The following notation will be used throughout this manuscript. 

\begin{notation}\label{basic notation}
Let $\mc X$ be a normal model for a semi-global field $F$.
If $P$ is a (not necessarily closed) point of the closed fiber $X$, let $R_P$ be the local ring of $\mc X$ at $P$; let $\wh R_P$ be its completion with respect to the maximal ideal $\frak m_P$; and let $F_P$ be the fraction field of $\wh R_P$.  In the case that $P$ is a closed point of $X$, the {\em branches} of $X$ at $P$ are the height one prime ideals of $\wh R_P$ that contain $t$, which we can also regard as the codimension one points of $\Spec(\wh R_P)$ that lie on the closed fiber.  The localization $R_\wp$ of $\wh R_P$ at a branch $\wp$ is a discrete valuation ring; we write $\wh R_\wp$ for its completion, and $F_\wp$ for the fraction field of $\wh R_\wp$.  The contraction of $\wp$ to $F_P$ determines an irreducible component $X_0$ of $X$, whose generic point $\eta$ has the property that $F_\eta \subset F_\wp$.  We then say that $\wp$ {\em lies on} $X_0$.  Note that the residue field
$k(\wp) := \wh R_\wp/\wp \wh R_\wp = R_\wp/\wp R_\wp$ is isomorphic to the fraction field of $\wh R_P/\wp$, and hence is a complete discretely valued field.

If $U$ is a nonempty connected affine open subset of $X$, 
then we write $R_U$ for the subring of $F$ consisting of rational functions that are regular at each point of $U$.  We let $\wh R_U$ be the $t$-adic completion of $R_U$; this completion is the {\em ring of formal functions along} $U$. 
This is an integral domain by \cite[Proposition~3.4]{HHK:refinements}, and we let $F_U$ be the fraction field of $\wh R_U$.  
If $P \in U \subseteq U'$, then $\wh R_{U'} \subseteq \wh R_U \subset \wh R_P$ and $F_{U'} \subseteq F_U \subset F_P$.  We say that a branch $\wp$ of $X$ at a closed point $P \in X$ {\em lies on} $U$ if it lies on a component of the closure $\bar U$.  The field $F_U$ is then contained in $F_\wp$.  If $U$ as above is affine, then 
the absolute value on the complete discretely valued field $k(\wp)$ restricts to an absolute value on $k(\bar U\red)$, where the superscript indicates the underlying reduced scheme.  Here $k(\wp) = \Frac(\wh{\mc O}_{\bar U\red,P})$ while $k(U\red) = k(\bar U\red) = \Frac(\mc O_{\bar U\red,P})$, and so $k(U\red)$ is dense in $k(\wp)$ under the above absolute value.
\end{notation}

To illustrate the above in a simple case, consider a smooth $T$-curve $\mc X$, such as the projective line over $T$, and pick 
a closed point $P \in X$.  Let $U \subset X$ be the complement of $P$ in $X$, and let $\wp$ be the unique branch of $X$ at $P$.  We then have containments of fields $F \subset F_P,\!F_U \subset F_\wp$, forming a diamond  
\[
\xymatrix @R=.7cm @C=.7cm  {
& F_\wp &\\
F_P \ar[ur] && F_U \ar[ul]\\
& F \ar[ul] \ar[ur] &
}
\]
with $F = F_P \cap F_U$ (see~\cite[Proposition~6.3]{HH:FP}).
Given a finite separable field extension of one of these four fields, we can ask whether it is induced by base change from an extension of the smaller fields.  More generally, we may have more complicated configurations of fields (see Notation~\ref{geom notn} below), but we can still ask this question.  Note that if a field extension $E$ of a larger field is shown to be induced by an \'etale algebra $A$ over a smaller field, then $A$ is automatically a (separable) field extension of the smaller field, because it is contained in $E$.  

Along these lines, the following results were proven
in \cite{CHHKPS 0cyc} (Propositions 2.3 and 2.4 there), 
in connection with obtaining local-global principles for zero-cycles on varieties over $F$.  These results in particular concern two of the four edges of the above diamond in the special case of the above simple example.

\begin{prop} \label{branch to point}
Let $F$ be a semi-global field with normal model $\mc X$ and closed fiber $X$. Let $P$ be a closed point of $X$, let $\wp$ be a branch of $X$ at~$P$, and
let $E_\wp$ be a finite separable field extension of $F_\wp$, say of degree $n$.  Then 
there exists a finite separable field extension $E_P$ of $F_P$ such that 
$E_P \otimes_{F_P} F_\wp\cong E_\wp$ as extensions of~$F_\wp$, and such that 
$E_P$ induces the trivial \'etale algebra $F_{\wp'}^{\oplus n}$ over $F_{\wp'}$ for every other branch
$\wp'$ at~$P$.
\end{prop}

\begin{prop} \label{component to global}
Let $F$ be a semi-global field with normal model $\mc X$ and closed fiber $X$. Let $U$ be a nonempty connected affine open subset of $X$, and let $E_U$ be a finite separable field extension of $F_U$.  Then there is a finite separable field extension $E$ of $F$ such that $E \otimes_F F_U \cong E_U$ as extensions of~$F_U$.
\end{prop}

These results raise the question of whether there are analogs with the roles of $P$ and $U$ interchanged; these would in particular treat the other two edges of the above diamond in that example.  Such analogs of Propositions~\ref{branch to point} and~\ref{component to global} do not hold in general; see Section~\ref{char p subsec} below for counterexamples.
But analogs do hold if $\cha(k)=0$, where as above $k$ is the residue field of the complete discretely valued field $K$:

\begin{prop} \label{branch to component}
Let $F$ be a semi-global field over a complete discrete valuation ring with residue field~$k$, and assume that $\cha(k)=0$. Let $\mc X$ be a normal model for $F$ with closed fiber $X$.
Let $U \subset X$ be a nonempty connected affine open subset, and 
let $\wp$ be a branch of $X$ lying on $U$.  Then
for every finite separable field extension $E_\wp$ of $F_\wp$ there is a finite separable field extension 
$E_U$ of $F_U$ such that $E_U \otimes_{F_U} F_\wp \cong E_\wp$ as $F_\wp$-algebras.
In fact there is a finite separable field extension 
$E$ of $F$ such that $E \otimes_F F_\wp \cong E_\wp$ as $F_\wp$-algebras.
\end{prop}

\begin{proof}
Since $F \subset F_U \subset F_\wp$, 
the first assertion follows from the second, by taking $E_U = E \otimes_F F_U$.  So it suffices to prove the second assertion.  By 
Proposition~\ref{component to global}, to prove that assertion it suffices to prove the first assertion for {\em some} choice of $U$ on which $\wp$ lies.

We first deal with the unramified part of the extension. 
Since $F_\wp$ is a complete discretely valued field, the maximal unramified extension $F_\wp'$ of 
$F_\wp$ contained in $E_\wp$ is also a complete discretely valued field.  
Moreover $E_\wp$ is totally ramified over $F_\wp'$.
Let $\wh R_\wp'$ be the integral closure of $\wh R_\wp$ in $F_\wp'$, say with maximal ideal $\wp'$.
The reduction $k(\wp') = \wh R_\wp'/\wp'$ is then a finite separable field extension of the 
complete discretely valued field $k(\wp)=\wh R_\wp/\wp\wh R_\wp$.  

Let $X_0$ be the irreducible component of $X$ on which $\wp$ lies, and let $U$ be a nonempty affine open subset of $X_0$ that does not meet any other irreducible component of $X$ and does not contain the point $P$ at which $\wp$ is a branch, and such that $U\red$ is regular.
Since $k(U\red)$ is dense in $k(\wp)$, by Krasner's Lemma there is a finite separable field extension $L$ of $k(U\red)$ that induces $k(\wp')$ over $k(\wp)$.  Thus $L = k(U')$, where $U'$ is a finite branched cover of~$U\red$.  After shrinking $U$, we may assume that $U' \to U\red$ is \'etale; i.e., $U' = \Spec(B)$ for some \'etale algebra $B$ over $k[U\red]$ whose fraction field is $L$.  
By \cite[I, Corollaire~8.4]{SGA1}, up to isomorphism there is a unique \'etale $\wh R_U$-algebra $\wh R_U'$ that lifts the \'etale $k[U\red]$-algebra $B = k[U']$.  

The base change of $\wh R_U'$ to $\wh R_\wp$ lifts the residue field extension $k(\wp')$ of $k(\wp)$.  But by \cite[I, Th\'eor\`em~6.1]{SGA1}, such a lift to a complete local ring is unique.  Hence 
the $\wh R_\wp$-algebra $\wh R_U' \otimes_{\wh R_U} \wh R_\wp$ is isomorphic to $\wh R_\wp'$, and so $\wh R_U'$ is a domain.  Thus
$F_U' := \wh R_U'  \otimes_{\wh R_U} F_U$ is a field that satisfies 
$F_U' \otimes_{F_U} F_\wp \cong F_\wp'$.  This completes the unramified step.

We now turn to the totally ramified part, and work explicitly.  With notation as above, let $\eta'$ be the generic point of $U'$, let $R_{\eta'}$ be the local ring of $\wh R_U'$ at $\eta'$, 
and let $\tau \in R_{\eta'} \subset F_U'$ be a uniformizer for the discrete valuation ring $R_{\eta'}$.  Thus $\tau$ is also a uniformizer for the complete
discrete valuation ring $\wh R_\wp'$, which contains $R_\eta$.
Let $\wh S_\wp$ be the integral closure of $\wh R_\wp'$ in $E_\wp$,
or equivalently of $\wh R_\wp$ in $E_\wp$.
Let $\til \wp$ be the maximal ideal of the complete
discrete valuation ring $\wh S_\wp$, 
and let $\sigma \in \wh S_\wp$ be a uniformizer at $\til \wp$.

Now $E_\wp$ is totally ramified over $F_\wp'$ along $\wp'$, say of degree $n=[E_\wp:F_\wp']$.
By the characteristic zero hypothesis, $\tau=\sigma^n v$ for some unit $v \in \wh S_\wp$;    
and the extension of 
residue fields $k(\wp') = \wh R_\wp'/\wp' \subseteq \wh S_\wp/{\til \wp}$
is an isomorphism.  So the image $\bar v \in \wh S_\wp/{\til \wp}$ of $v\in \wh S_\wp$
may be regarded as a non-zero element in the complete discretely valued field $k(\wp')$.

Let $C,C'$ be the regular connected projective curves containing $U\red,U'$, respectively.  Thus   
$C'$ is a branched cover of $C$; and there is a birational map $C \to X_0\red$ which is an isomorphism over $U\red$.  
Since $\wp$ is a branch of $X$ at $P$,
it is also
a branch of $C$ at a closed point $Q \in C$ that lies over $P \in \bar U$ on $X$.  Moreover $\wp'$ is a branch of $C'$ at a point $Q'$ that maps to $Q$.  Also, 
$k(\wp)$ is the fraction field of 
the complete local ring $\wh {\cal O}_{C,Q}$ of $C$ at $Q$, and 
$k(\wp')$ is the fraction field of $\wh {\cal O}_{C',Q'}$.

Let $\bar\pi \in {\cal O}_{C',Q'} \subset k(C')$ be a uniformizer of the local ring of $C'$ at $Q'$.  Thus $\bar\pi$ is also 
a uniformizer of the complete local ring $\wh {\cal O}_{C',Q'}$, which is the valuation ring of $k(\wp')$.  
So we may write $\bar v = \bar\pi^r \bar u$ for some integer $r$ and some unit 
$\bar u \in \wh{\cal O}_{C',Q'}$.
Since ${\cal O}_{C',Q'}$ is $\bar \pi$-adically dense in $\wh {\cal O}_{C',Q'}$, there is a unit
$\bar w \in {\cal O}_{C',Q'} \subset k(C')$ such that $\bar w \equiv \bar u$ modulo $\bar\pi$.  Thus
$\bar u = \bar a\bar w$ for some $\bar a \equiv 1 \!\!\mod \bar\pi$ in $\wh {\cal O}_{C',Q'}$.  

Since the residue field of the complete discrete valuation ring $\wh {\cal O}_{C',Q'}$ has 
characteristic zero, by Hensel's Lemma
there is a unit $\bar b \in \wh {\cal O}_{C',Q'} \subset k(\wp')$ 
such that $\bar b^n=\bar a$.  So $\bar v = \bar \pi^r \bar w \bar b^n$.  
Let $\pi,w \in R_{\eta'} \subset F_U'$ be lifts of $\bar\pi,\bar w \in k(C') = R_{\eta'}/\eta'R_{\eta'}$, and let
$b \in \wh R_\wp'$ be a lift of $\bar b \in k(\wp') = \wh R_\wp'/\wp'$.
Thus $v=\pi^r w b^n c \in \wh S_\wp$, for some $c \equiv 1 \!\!\mod \sigma$ in $\wh S_\wp$.
Again by Hensel's Lemma, there exists $d \in \wh S_\wp$
such that $c=d^n$.  Note that $\pi,w$ are units in $R_{\eta'}$ and
hence in $\wh R_\wp'$; and that $b,d$ are units  
in $\wh S_\wp$.

Thus $\tau\pi^{-r} w^{-1} \in R_{\eta'} \subset \wh R_\wp'$ is a uniformizer for $\wh R_\wp'$; 
$bd\sigma$ is a uniformizer for $\wh S_\wp$; and 
$(bd\sigma)^n=\tau\pi^{-r} w^{-1}$.  
Since $E_\wp = \Frac(\wh S_\wp)$ is of degree $n$ over $F_\wp' = \Frac(\wh R_\wp)$,
it follows that $E_\wp$ is generated over $F_\wp'$ by 
$bd\sigma$; i.e.,  
$E_\wp \cong F_\wp'[Y]/(Y^n- \tau\pi^{-r} w^{-1})$.  
Since $\tau\pi^{-r} w^{-1} \in R_{\eta'} \subset F_U'$, we may
consider the finite separable $F_U'$-algebra $E_U = F_U'[Y]/(Y^n- \tau\pi^{-r} w^{-1})$.  
We then have $E_U \otimes_{F_U'} F_\wp' \cong E_\wp$, and hence $E_U$ is a field.
Since $F_\wp' \cong F_U' \otimes_{F_U} F_\wp$, it follows that
$E_U \otimes_{F_U} F_\wp \cong E_\wp$, as desired.  
\end{proof}

The key ingredient in the proof of the next theorem, in addition to the propositions above, is patching over fields, as in \cite{HH:FP} and \cite{HHK}.  Recall that a scheme $V$ is {\em unibranched} at a point $P$ if the complete local ring $\wh{\mc O}_{V,P}$ has a unique minimal prime; or equivalently, $\Spec(\wh{\mc O}_{V,P})$ is irreducible.

\begin{thm} \label{point to global}
Let $F$ be a semi-global field over a complete discrete valuation ring with residue field $k$, and  assume that $\cha(k)=0$.  Let $\mc X$ be a normal model of $F$ with closed fiber~$X$. Let $P$ be a closed point of $X$, and assume that 
every irreducible component of $X$ that contains $P$ is unibranched at $P$.
If $E_P$ is a finite separable field extension of $F_P$, 
then there is a finite separable field extension $E$ of $F$ such that $E \otimes_F F_P \cong E_P$ as extensions of~$F_P$.
\end{thm}

\begin{proof}
Choose a finite set $\mc P$ of closed points of $X$ that contains $P$ and also all the points where distinct irreducible components of $X$ meet; and let $\mc U$ be the set of connected components of the complement of $\mc P$ in $X$.  Thus each $U \in \mc U$ is affine and meets just one irreducible component of $X$.
Let $\mc U_P$ be the subset of $\mc U$ consisting of those $U \in \mc U$ whose closures contain $P$;
and let $\mc P_P$ be the subset of $\mc P$ consisting of points other than $P$ that lie in the closure of some element of $\mc U_P$.  
Blowing up $\mc X$ at the points of $\mc P_P$ produces a new model, but does not change $F$ or $F_P$.
After doing so (possibly several times), we may assume that the following two conditions hold:
\begin{enumerate}
\renewcommand{\theenumi}{\roman{enumi}}
\renewcommand{\labelenumi}{(\roman{enumi})}
\item \label{prop cond i} If $P' \in \mc P_P$ lies on the closure of $U \in \mc U_P$, then $\bar U$ is unibranched at $P'$.
\item  \label{prop cond ii} For each $P' \in \mc P_P$, there is a unique $U \in \mc U_P$ whose closure contains $P'$.
\end{enumerate}

For each branch $\wp$ at $P$, $E_\wp := E_P \otimes_{F_P} F_\wp$
is a finite direct product of finite separable field extensions $E_{\wp, i}$ of $F_\wp$.
By Proposition~\ref{branch to component}, for each branch $\wp$ at $P$ and each $i$, there is
a finite separable field extension $E_{U,i}$ of $F_U$ such that 
$E_{U,i} \otimes_{F_U} F_\wp \cong E_{\wp,i}$, where $\wp$ lies on $U \in \mc U$.  For each $U \in \mc U$
whose closure contains $P$, let $E_U$ be the direct product of the fields $E_{U,i}$, ranging over $i$. 
This is well defined, for each $U \in \mc U_P$, because of the assumption on being unibranched at $P$.  For each branch $\wp$ along any $U \in \mc U_P$, let $E_\wp = E_U \otimes_{F_U} F_\wp$.  (For the branches at $P$, this agrees with the previous definition of $E_\wp$.)

By conditions (\ref{prop cond i}) and (\ref{prop cond ii}), for each $P' \in \mc P_P$ there is a unique $U \in \mc U_P$ whose closure contains $P'$; and there is a unique branch $\wp $ at $P'$ along $U$.  For such a triple $P',U,\wp$, 
applying Proposition~\ref{branch to point} to each direct factor of $E_\wp$ provides a finite \'etale 
$F_{P'}$-algebra $E_{P'}$ such that $E_{P'} \otimes_{F_{P'}} F_\wp \cong E_\wp$ and such that 
$E_{P'} \otimes_{F_{P'}} F_{\wp'}$ is the trivial   
\'etale $F_{\wp'}$-algebra of degree $n:=[E_P:F_P]$ for every branch $\wp'$ at $P'$ other than $\wp$.  For every 
$U \in \mc U$ that is not in $\mc U_P$, let $E_U$ be the trivial \'etale $F_U$-algebra of degree $n$.  
Similarly, for every $P' \in \mc P$ that does not lie in $\mc P_P \cup \{P\}$, let $E_{P'}$ be the trivial 
\'etale $F_{P'}$-algebra of degree $n$; and for every branch $\wp B$ at a point of $\mc P$ 
that is not in $\mc P_P \cup \{P\}$, let $E_\wp$ be the trivial \'etale
$F_\wp$-algebra of degree $n$.  Then for every branch $\wp $ at a point $P' \in \mc P$ lying on some $U \in \mc U$, we have isomorphisms $E_{P'} \otimes_{F_{P'}} F_\wp \cong E_\wp \cong E_U \otimes_{F_U} F_\wp$. We then conclude by Theorem 7.1 of \cite{HH:FP}.
\end{proof}

\begin{remark} \label{char p descent rks}
\begin{enumerate}
\renewcommand{\theenumi}{\alph{enumi}}
\renewcommand{\labelenumi}{(\alph{enumi})}  
\item \label{wild ram rk}
The hypothesis that $\cha(k)=0$ in Proposition~\ref{branch to component} was used in order to avoid wild ramification and inseparable residue field extensions; and it was used in Theorem~\ref{point to global} in order to be able to rely on Proposition~\ref{branch to component}.  Otherwise the proofs carry over to characteristic $p>0$.  For example, if $E_\wp/F_\wp$ is a Galois field extension of degree prime to $p$, then all ramification is tame, and the conclusion of Proposition~\ref{branch to component} holds.  Similarly, the conclusion of Theorem~\ref{point to global} holds for a Galois field extension $E_P/F_P$ of degree prime to $p$, provided that the condition on being unibranched at $P$ is satisfied.  But these two propositions fail in general for wildly ramified extensions, as shown in Section~\ref{char p subsec}.
\item \label{unibranch rk}
In Theorem~\ref{point to global}, the hypothesis on being unibranched is essential.  Namely, suppose instead 
that $\wp_1, \wp_2$ are distinct branches of an irreducible component $X_0$ of $X$ at $P$; these are height
one primes in the Noetherian normal domain $\wh R_P$.  By the corollary in \cite[Section~VII.3]{Bo:CA}, $\wh R_P$ is a Krull domain; so by \cite[Proposition~VII.5.9 and Theorem~VII.6.3]{Bo:CA}, 
there exists $f \in \wh R_P$ that is a uniformizer at $\wp_1$ but does not lie in $\wp_2$.  
Let $\ell$ be a prime unequal to $\cha k$; let $E_P$ be the finite separable extension of $F_P$ given by adjoining an $\ell$-th root of $f$; this is ramified over $\wp_1$ but not over $\wp_2$.  Write
$E_{\wp_i} = E_P \otimes_{F_P} F_{\wp_i}$.  Then $E_{\wp_1}/F_{\wp_1}$ is ramified over $\wp_1$ but $E_{\wp_2}/F_{\wp_2}$ is not ramified over $\wp_2$.  Let $\eta$ be the generic point of $X_0$.  A uniformizer of $R_\eta$ is also a uniformizer of $R_{\wp_i}$ for $i=1,2$.  Thus if 
$E/F$ is a degree $\ell$ field extension, then $E_i := E \otimes_F F_{\wp_i}$ is ramified over $\wp_i$ if and only if $E/F$ is ramified over $\eta$.  Thus $E/F$ cannot induce both $E_{\wp_1}/F_{\wp_1}$ and $E_{\wp_2}/F_{\wp_2}$.  But $E_P/F_P$ induces $E_{\wp_1}/F_{\wp_1}$ and $E_{\wp_2}/F_{\wp_2}$.   Then $E/F$ cannot induce $E_P/F_P$. 
\item \label{refs rk}
A result related to Theorem~\ref{point to global} was proven in \cite[Lemma~5.2]{HS}.  That assertion was stated in the equal characteristic case, and it permitted the characteristic to be non-zero.  By that result, if a finite Galois extension of $k((x,t))$ is unramified over the ideal $(t)$ of $k[[x,t]]$, then it is induced by a finite Galois extension of the function field $F =k((t))(x) \subset k((x,t))$ of the $k((t))$-line.  Moreover, by a change of variables in $k((x,t))$, 
the condition on being unramified over $(t)$ can be dropped (see also \cite[Lemma~3.8]{HHK:Weier}); but this makes it impossible to specify $F$ in advance as a subfield of $k((x,t))$, unlike in the above results that restrict to characteristic zero.
\end{enumerate} 
\end{remark}

\subsection{Application to a local-global principle for zero-cycles} \label{lgp applic sec}

As in \cite{HHK}, we also use the following notation:

\begin{notation} \label{geom notn}
Let $F$ be a semi-global field with normal model $\mc X$, and let $X$ denote the closed fiber. Let $\mc P$ be a finite set of closed points of $X$ that meets each irreducible component of $X$.  We then let $\mc U$ be the set of connected components of the complement of $\mc P$ in $X$, and we let $\mc B$ be the set of branches of $X$ at points of $\mc P$.  
\end{notation}

Recall that given a variety $V$ over a field $k$, the \textit{index} (resp.\ \textit{separable index}) of $V$ is the greatest common divisor of the degrees of the finite (resp.\ finite separable) field extensions of $k$ over which $V$ has a rational point.  This is the same as the smallest positive degree of a zero-cycle (resp.\ separable zero cycle) on $V$.

As in \cite{CHHKPS 0cyc}, given a collection $\Omega$ of overfields of $F$ and an $F$-scheme $Z$, we say that $(Z,\Omega)$ {\em satisfies a local-global principle for rational points} if the following holds: $Z(F)\neq \varnothing$ if and only if $Z(L)\neq \varnothing$ for all $L \in \Omega$. In particular, we will consider the collection of overfields $\Omega_{\mc X,\mc P}$ consisting of the overfields $F_P, F_U$ for $\mc P$ and $\mc U$ as in Notation~\ref{geom notn}.

If $Z$ is a torsor under a connected and rational linear algebraic group over $F$, then $(Z,\Omega_{\mc X,\mc P})$ satisfies a local-global principle for rational points, for any choice of $\mc P$ and $\mc U$ as above by \cite[Theorem~3.7]{HHK}; see~\cite[Corollary~3.10]{CHHKPS 0cyc} for a further discussion.

In Corollaries 3.5 and 3.6 of \cite{CHHKPS 0cyc}, we showed:

\begin{prop} \label{LGP patches}
Let $F$ be a semi-global field with normal model $\mc X$, and let $X$ be the closed fiber. Let 
$Z$ be an $F$-scheme of finite type, and choose $\mc P$ and $\mc U$ as in Notation~\ref{geom notn}. Assume that for every finite separable extension $E/F$, $(Z_E,\Omega_{\mc X_E,\mc P_E})$ satisfies a local-global principle for rational points, where $\mc X_E$ denotes the normalization of $\mc X$ in $E$ and $\mc P_E$ denotes the preimage of $\mc P$ under the normalization map. Then 
\begin{enumerate}
\renewcommand{\theenumi}{\alph{enumi}}
\renewcommand{\labelenumi}{(\alph{enumi})}  
\item
The prime numbers that divide the separable index of $Z$ are precisely those that divide the separable index of some $Z_{F_\xi}$, where $\xi$ ranges over $\mc P \cup \mc U$.  In particular, the separable index of $Z$ is equal to one if and only if the separable index of each $Z_{F_\xi}$ is equal to one.
\item
If $\cha K=0$, or if $Z$ is regular and generically smooth, then the assertion also holds with the separable index replaced by the index.
\end{enumerate}
\end{prop}

Proposition~\ref{component to global} and Theorem~\ref{point to global} together
yield a strengthening of Proposition~\ref{LGP patches} in the equal characteristic zero case:

\begin{cor} \label{prod bound}
In the situation of Proposition~\ref{LGP patches}, suppose that
$\cha(k)=0$ and that 
each irreducible component $X_0$ of the closed fiber $X$ of $\mc X$ is unibranched at each point $P \in \mc P$ on $X_0$. 
Then the index of $Z$ divides the product of the indices of $Z_{F_\xi}$, for 
$\xi \in \mc P \cup \mc U$.
\end{cor}

\begin{proof} First consider field extensions $E_\xi/F_\xi$ for $\xi \in \mc P \cup \mc U$, say of degree $d_\xi$, such that $Z_{F_\xi}$ has an $E_\xi$-point. 
By
Proposition~\ref{component to global} and Theorem~\ref{point to global}, there are finite field extensions $A_\xi/F$ of degree $d_\xi$ that induce $E_\xi/F_\xi$, for each $\xi$.  
Thus $Z(A_\xi \otimes_F {F_\xi}) = Z(E_\xi)$ is nonempty for each $\xi$.
Let $A = \bigotimes_F A_\xi$, where the tensor product is taken over all $\xi \in \mc P \cup \mc U$.
Then for each $\xi$, $Z(A \otimes_F F_\xi)$ is nonempty.  Now $A$ is an 
\'etale algebra over $F$, and so it is the direct product of finite field extensions $A_i/F$.  
Note that $\sum_i [A_i:F] = \dim_F(A) = \prod d_\xi$.
Also, for each $i$, $Z(A_i \otimes_F F_\xi)$ is nonempty for every $\xi$.  

Let $\mc X_i$ be the normalization of $\mc X$ in $A_i$; this is a normal model of the semi-global field $A_i$, equipped
with a finite morphism $\mc X_i \to \mc X$.  Let $\mc P_i$, $\mc U_i$ consist of the inverse images of the elements of $\mc P$, $\mc U$ under this morphism, respectively.  As in Notation~\ref{branch to component}, we then have associated fields  
$(A_i)_{P'}, (A_i)_{U'}$ for $P' \in \mc P_i$ and $U' \in \mc U_i$.  For each $P \in \mc P$, $A_i \otimes_F F_P$ is the direct product of the fields $(A_i)_{P'}$, where
$P'$ runs over the points of $\mc P_i$ that lie over $P$; and similarly for each $U \in \mc U$.  Hence for each $\xi'\in \mc P_i \cup \mc U_i$, $Z((A_i)_{\xi'})$ is nonempty. By the local-global assumption, this implies $Z(A_i)$ is nonempty, for each $i$.  

Let $I$ (resp.~$I_\xi$) be the ideal in $\mbb Z$ generated by the degrees of closed points on $Z$ (resp.\ on $Z_{F_\xi}$); or equivalently, generated by the index of $Z$ (resp.\ of $Z_{F_\xi}$).  Since $\prod d_\xi =\sum_i [A_i:F]$ above, it follows that $\prod d_\xi \in I$; i.e.\ 
$\prod I_\xi \subseteq I$.  The asserted conclusion follows.
\end{proof}

Note that the above
bound is not sharp, since by enlarging the set $\mc P$ we can in general increase the product of the local indices.  It seems reasonable to ask whether, under the above hypotheses, the index of $Z$ is equal to the least common multiple of the indices of $Z_{F_\xi}$, for 
$\xi \in \mc P \cup \mc U$.  We do not know of any counterexamples.

\subsection {Failure of descent in characteristic $p$} \label{char p subsec}

Proposition~\ref{branch to component} and Theorem~\ref{point to global} fail if $\cha(k) \ne 0$, as shown in Proposition~\ref{counterexample prop} below. First we state a lemma.

\begin{lemma} \label{Gal over alg cl}
Let $D/L$ be a finite field extension, and let $\Lambda \supseteq L$ be a field in which $L$ is algebraically closed.  If 
$D \otimes_L \Lambda$ is a Galois field extension of $\Lambda$, then $D/L$ is Galois. 
\end{lemma}

\begin{proof}
First note that if $D \otimes_L \Lambda$ is Galois, and hence separable, then $D/L$ is separable.
Let $\wh D$ be the Galois closure of $D/L$, and write
$G = \Gal(\wh D/L)$ and $H = \Gal(\wh D/D)$.  Since $L$ is algebraically closed in $\Lambda$ whereas $D$ and $\wh D$ are each separable over $L$, it follows that 
$\Lambda$ is linearly disjoint over $L$ from both $D$ and $\wh D$.
Let $\Delta = D \otimes_L \Lambda$ and $\wh \Delta = \wh D \otimes_L \Lambda$
Thus $\wh \Delta$ is a Galois field extension of $\Lambda$ with $\Gal(\wh \Delta/\Lambda) = G$, and 
$\Gal(\wh \Delta/\Delta) = H$.  Since $\Delta/\Lambda$ is Galois, it follows that $H$ is normal in $G$; hence $D/L$ is Galois.
\end{proof}

\begin{lemma} \label{AS descent la}
Let $k_2/k_1$ be a field extension in characteristic $p>0$ such that the algebraic closure of $k_1$ in $k_2$ is separable over $k_1$,
and write $F_i = k_i((t))$ for $i=1,2$.
Let $\alpha \in k_2^\times$, and let $E: = F_2[Y]/(Y^p-Y-\alpha/t)$.  Then $E$ is a
degree $p$ Galois field extension of $F_2$.  Moreover:
\renewcommand{\theenumi}{\alph{enumi}}
\renewcommand{\labelenumi}{(\alph{enumi})}  
\begin{enumerate}
\item \label{AS Galois descent} 
If $\alpha$ is not in $k_1$, then $E$ is not induced by a degree $p$ Galois field extension of $F_1$.  
\item \label{AS genl descent}
If $\alpha$ is not algebraic over $k_1$, then $E$ is not induced by any degree $p$ field extension of $F_1$.
\end{enumerate}
\end{lemma}

\begin{proof}
Since $F_2$ is of characteristic $p$, and since $\alpha/t$ is not of the form $c^p-c$ for any $c \in F_2$, it 
follows that $E$ is a degree $p$ Galois field extension of $F_2$.

For part (\ref{AS Galois descent}), assume $\alpha \not \in k_1$, and suppose that $E$ is induced by a degree $p$ Galois field extension of $F_1$.  We may then write that extension of $F_1$ as $F_1[W]/(W^p-W-\beta)$ for some $\beta \in F_1$, with $\alpha t^{-1} - \beta = \gamma^p-\gamma$ for some $\gamma \in F_2$.  
Projecting both sides of this equality onto
the $k_2$-vector subspace of $F_2$ spanned by $t^{-1},t^{-p}, t^{-p^2},\dots$, we may assume that $\beta$ and $\gamma$ lie in that subspace.  Write 
$\beta = \sum_{i=0}^n a_it^{-p^i}$ with $a_i \in k_1$.  
Then $\alpha^{p^n}t^{-p^n}  - (\sum_{i=0}^n a_i^{p^{n-i}})t^{-p^n}=\delta^p-\delta$ for some $\delta \in F_2$; so $\alpha^{p^n}= \sum_{i=0}^n a_i^{p^{n-i}} \in k_1$.  Thus $\alpha \in k_2$ is purely inseparable over $k_1$, and hence lies in $k_1$; a contradiction.

For part (\ref{AS genl descent}), assume that $\alpha$ is not algebraic over $k_1$.  Let $k_1'$ be the algebraic closure of $k_1$ in $k_2$; this is a separable extension of $k_1$ not containing $\alpha$.  Let $F_1' = k_1'((t))$.
The algebraic closure of $F_1'$ in $F_2$ is an unramified extension of complete discretely valued fields with trivial residue field extension, and so this extension of $F_1'$ is trivial; i.e., $F_1'$ is the algebraic closure of $F_1$ in $F_2$.  If $E_1/F_1$ is a degree $p$ field extension such that
$E_1 \otimes_{F_1} F_2 = E$, then Lemma~\ref{Gal over alg cl} asserts that $E_1 \otimes_{F_1} F_1'$ is a Galois field extension of $F_1'$.  
But this field extension induces $E/F_2$.  By applying part (\ref{AS Galois descent}) to the extension $k_2/k_1'$, we obtain a contradiction.
\end{proof}

\begin{example}
The transcendentality condition in part~(\ref{AS genl descent}) of Lemma~\ref{AS descent la} is necessary.  For example, 
let $\kappa$ be a field of characteristic three; let $k_1=\kappa(x)$; and let $k_2=\kappa(u)$ where $u^2=x$.  Let $\alpha = u$, which lies in $k_2$ and is algebraic over $k_1$ but does not lie in $k_1$.  
Write $F_i = k_i((t))$ for $i=1,2$, and let $E/F_2$ be the $3$-cyclic Galois extension given by $E=F_2[Y]/(Y^3-Y-u/t)$.  Then $E$ is not induced by a $3$-cyclic Galois extension of $F_1$, but it is induced by the degree three non-Galois extension $E_0/F_1$ given by $E_0=F_1[W]/(W^3+W^2+W-x/t^2)$.  
Here $E/F_1$ is an $S_3$-Galois field extension, and $E_0$ is the fixed field of the order two subgroup of $S_3$ generated by the involution taking $u$ to $-u$ and 
taking $Y$ to $-Y$.  As an extension of $F_1$, the field $E_0$ is generated by the element $W=Y^2$.
\end{example}

We will apply Lemma~\ref{AS descent la} in the situation in which $k_1$ is the fraction field of a 
Dedekind domain $D$, and $k_2$ is the fraction field of the completion of $D$ at a maximal ideal.  By Artin's Approximation Theorem (Theorem~1.10 of \cite{artin}), the separability hypothesis in Lemma~\ref{AS descent la} will hold if $D$ is excellent; e.g., if it is the coordinate ring of a curve over a field of characteristic $p$.
In that same situation, we next prove 
a mixed characteristic analog of Lemma~\ref{AS descent la}.

\begin{lemma} \label{Kummer descent la}
Let $k_1$ be the fraction field of a characteristic $p$ excellent Dedekind domain $D$; let $k_2$ be the fraction field of the completion $\wh D$ of $D$ at a maximal ideal $\frak m$, and 
let $D'$ be the algebraic closure of $D$ in $\wh D$.  
Let $R_1 \subseteq R_2$ be an extension of complete discrete valuation rings of mixed characteristic $(0,p)$ having residue field extension $k_2/k_1$, and let $F_i$ be the fraction field of $R_i$.
Let $x \in R_1$ be an element whose reduction $\bar x \in k_1$ is a uniformizer for $D$ at $\frak m$, and let $g \in R_2$ be an element whose reduction $\bar g \in k_2$ lies in 
$\wh D\smallsetminus D'$.  
Then $E = F_2[Y]/(Y^p-g^p-x)$ is a degree $p$ field extension of $F_2$ that is not induced by any degree $p$ field extension of $F_1$, nor by any degree $p$ Galois extension of the algebraic closure $F_1'$ of $F_1$ in $F_2$. 
\end{lemma}
 
\begin{proof}
It suffices to show that the extension $E(\zeta_p)/F_2(\zeta_p)$ is not induced by any degree $p$ field extension of $F_1(\zeta_p)$, nor by any degree $p$ Galois extension of $F_1'(\zeta_p)$.  So 
after replacing $F_i$ by $F_i(\zeta_p)$ for $i=1,2$ (which does not change $k_i$), we may assume that $\zeta_p \in R_i \subset F_i$.
Now the residue class $\bar f \in k_2$ of $f  := g^p + x\in R_2$ 
is not a $p$-th power because the residue class of $x$ is not a $p$-th power there.  So $f$ is also not a $p$-th power in $F_2$, and $E$ is a degree $p$ Galois field extension of $F_2$ (viz.\ a Kummer extension).  

We claim that for every $e \in k_2^\times$, $\bar fe^p$ does not lie in the algebraic closure $k_1'$ of $k_1$ in $k_2$.  Since the residue field of
$F_1'$ is $k_1'$, the claim implies that $f \til e^p \not\in F_1'$
for every $\til e \in F_2^\times$.  Hence $f \in F_2^\times$ is not in the same $p$-th power class as any element of $F_1'$.  Thus $E$ 
is not induced by any degree $p$ Galois field extension of $F_1'$.
By Lemma~\ref{Gal over alg cl}, $E$ is also not induced by any degree $p$ field extension of $F_1$.

To prove the claim (and therefore the assertion), suppose otherwise; i.e., $\bar fe^p \in k_1'$ for some $e \in k_2^\times$.  After multiplying $e$ by some non-zero element of $D'$, we may assume that $\bar fe^p$ is equal to some element $h \in D'$.  Since $f = g^p + x$, the elements $\bar ge, e$ satisfy the polynomial equation $Y_1^p + \bar xY_2^p - h=0$ over $D'$.  By Artin's Approximation Theorem, 
$D'$ is the henselization of $D$ at $\frak m$, and 
there exist elements $\bar g', e' \in D'$ such that $\bar g'e', e'$ are solutions
to the above equation, with $e' \ne 0$.  
So $(\bar g e)^p +\bar x e^p = (\bar g' e')^p +\bar x e'^p$. 
If $e\ne e'$ then $\bar x=(\bar g e - \bar g' e')^p/(e '-e)^p$, which is a contradiction since $\bar x$ is not a $p$-th 
power in $D$.  So $e=e' \in D'$, and $\bar g^p = (h - \bar xe^p)/e^p \in D'$,
using that $(\bar g e)^p +\bar x e^p=h$.  Since $D'$ is algebraically closed in $\wh D$, it follows that $\bar g \in D'$, which is a contradiction.
\end{proof}

\begin{prop} \label{counterexample prop}
Let $F$ be a semi-global field over a complete discrete valuation ring with residue field $k$, and assume that $\cha k=p>0$. Let $\mc X$ be a normal model for $F$, and let $X$ denote its closed fiber.
Let $P$ be a closed point of $X$ lying on an irreducible component $X_0$, let $\wp$ be a branch of $X_0$ at $P$, and let $U$ be a nonempty connected affine open subset of $X$ that meets $X_0$ but does not contain $P$.
\renewcommand{\theenumi}{\alph{enumi}}
\renewcommand{\labelenumi}{(\alph{enumi})}  
\begin{enumerate}
\item \label{char p counterex}
There is a degree $p$ Galois field extension $E_\wp$ of $F_\wp$ 
that is not induced by any degree $p$ field extension of $F_U$, or even a degree $p$ field extension of the algebraic closure of $F_U$ in $F_\wp$.

\item \label{mixed char counterex}
There is a degree $p$ Galois field extension $E_P$ of $F_P$ 
that is not induced by any degree $p$ field extension of $F$, or even a degree $p$ field extension of the algebraic closure of $F$ in $F_P$.
\end{enumerate}
\end{prop}

\begin{proof}
Let $\eta$ be the generic point of $X_0$,
and let $F_\eta'$ be the algebraic closure of $F_\eta$ in $F_\wp$.
Note that $F_\eta$ contains $F_U$ since $\eta \in U$, and so
$F_\eta'$ contains the algebraic closure of $F_U$ in $F_\wp$. 
We will show the following statement, which implies both parts of the proposition:
There is a degree $p$ Galois field extension $E_P$ of $F_P$ such that 
$E_\wp:=E_P \otimes_{F_P} F_\wp$ is a degree $p$ Galois field extension of $F_\wp$
that is not induced by any degree $p$ field extension of $F_\eta'$.

Let $\til X_0$ be the normalization of $X_0$ and let $\til \wp$ be the branch on $\til X_0$ lying over $\wp$; this is the unique branch of $\til X_0$ at some point $\til P$ of $\til X_0$ lying over $P$.  
The residue field $k(\wp)$ of $\wh R_\wp$ at $\wp$ is 
also the residue field of $R_\wp \subset \wh R_P$ at $\wp$; it is also
the fraction field of the complete local ring 
$\wh{\mc O}_{\til X_0,\til P}$ of $\til X_0$ at $\til P$.

By Artin's Approximation Theorem (Theorem~1.10 of \cite{artin}), 
the henselization $\mc O_{\til X_0,\til P}\h$ of $\mc O_{\til X_0,\til P}$ is algebraically closed in the completion $\wh{\mc O}_{\til X_0,\til P}$.  
So the algebraic closure of $k(\til X_0) = k(X_0) = k(\eta)$ in $k(\wp)$
is the fraction field of $\mc O_{\til X_0,\til P}\h$, which is separable over $k(\eta)$.
Let $\alpha \in k(\wp)$ be an element that is transcendental over 
$k(X_0)$; i.e., does not lie in the fraction field of $\mc O_{\til X_0,\til P}\h$.

First consider the case in which $\cha K = p$.   
That is, $K$ is a complete discretely valued field of equal characteristic $p$, hence the form $k((t))$,
with $F_\eta = k(\eta)((t))$ and $F_\wp = k(\wp)((t))$.
We will regard $k(\wp)$ as contained in $R_\wp$, and hence in $F_P$ and $F_\wp$;
and in particular we will regard $\alpha$ as an element of those fields.
Let $E_P$ be the degree $p$ Galois field extension of $F_P$ given by adjoining a root of 
$Y^p-Y-\alpha/t$, and let $E_\wp = E_P \otimes_{F_P} F_\wp$.  
We now apply Lemma~\ref{AS descent la}, taking $k_1$ equal to
the algebraic closure of $k(\eta)$ in $k(\wp)$, taking
$k_2=k(\wp)$, and taking $E=E_\wp$.  The lemma then says that $E_\wp$ has the asserted property, and this proves the result in the equal characteristic case.

Next, consider the case in which $\cha K = 0$.  Let $D$ be the local ring of $\til X_0$ at $\til P$, 
and let $\wh D = \wh{\mc O}_{\til X_0,\til P}$ be its completion at the point $\til P$.  Let $x$ be 
an element in the local ring of $\mc X$ at $\eta$ whose image $\bar x$ in the residue field $k(\eta) =k(\til X_0)$ is a uniformizer of $\til X_0$ at $\til P$.  Let $g \in R_\wp \subset \wh R_P$ be a lift of $\alpha \in k(\wp)$.  Then $E_P = F_P[Y]/(Y^p - g^p - x$) has the asserted property, by applying Lemma~\ref{Kummer descent la}, where we take $R_1=\wh R_\eta$, $R_2=\wh R_\wp$, and $E = 
E_\wp = E_P \otimes_{F_P} F_\wp$.
\end{proof}

Geometrically, the above proposition asserts in particular that if $\cha(k)=p>0$, then there is a degree $p$ Galois branched cover of $\Spec(\wh R_P) =\Spec(\wh{\mc O}_{\mc X,P})$ that is not induced by any degree $p$ branched cover of $\mc X$.

\section{Absolute Galois groups} \label{Gal gps sec}

\subsection{Injectivity of local-global maps on Galois groups} \label{inj subsec}

Given a finite group $G$, a field $L$, and a separable closure $L^\sep$ of $L$, 
homomorphisms $\Gal(L) := \Gal(L^\sep/L) \to G$ (which we require to be continuous) are in bijection with pairs consisting of a $G$-Galois \'etale $L$-algebra $E/L$ and an $L$-algebra homomorphism $E \to L^\sep$.  Here, an epimorphism $\phi:\Gal(L) \to G$ corresponds to the
$G$-Galois field extension $(L^\sep)^{\ker(\phi)}/L$ together with its inclusion into $L^\sep$.  A general homomorphism $\phi$ with image $H \subseteq G$ determines an $H$-Galois field extension $E_0/L$.  The induced $G$-Galois \'etale $L$-algebra $E
= \Ind_H^G(E_0)$ is equipped with a projection map $\Ind_H^G(E_0) \to E_0$,
and consequently a map to $L^\sep$ given by $E \to E_0 =
(L^\sep)^{\ker(\phi)} \subset L^\sep$.  (For further discussions, see \cite[Exp.~V]{SGA1} and \cite[Proposition~18.17]{KMRT}.)

Consider an inclusion of fields $L \subseteq E$.  If we pick a separable closure $E^\sep$ of $E$, then the 
separable closure of $L$ in $E^\sep$ is a separable closure $L^\sep$ of $L$ in the absolute sense.  There
is then an induced group homomorphism between the absolute Galois groups of $E$ and $L$; i.e., from 
$\Gal(E) := \Gal(E^\sep/E)$ to $\Gal(L) := \Gal(L^\sep/L)$.   This is a special case of the fact that a morphism
of pointed schemes $(V,v) \to (W,w)$ induces a homomorphism $\pi_1(V,v) \to \pi_1(W,w)$ between their 
\'etale fundamental groups.  In our situation, there is the following result about the homomorphism $\Gal(E) \to \Gal(L)$:

\begin{lemma} \label{Gal map la}
In the above situation, the map $\Gal(E) \to \Gal(L)$ induced by the inclusion $L^\sep \subseteq E^\sep$ factors as $\Gal(E) \twoheadrightarrow \Gal(EL^\sep/E) \hookrightarrow \Gal(L)$, and its image is  
$\Gal(E \cap L^\sep)$.  Thus the map is injective if and only if $E^\sep=EL^\sep$, and it is surjective if and only if $L$ is separably closed in $E$.
\end{lemma}

This lemma is a special case of results on Galois categories in~\cite[Exp.~V.6]{SGA1}, with the injectivity and surjectivity assertions respectively following from Corollaire~6.8 and Proposition~6.9 there.  The lemma also follows directly from this diagram of fields and inclusions:
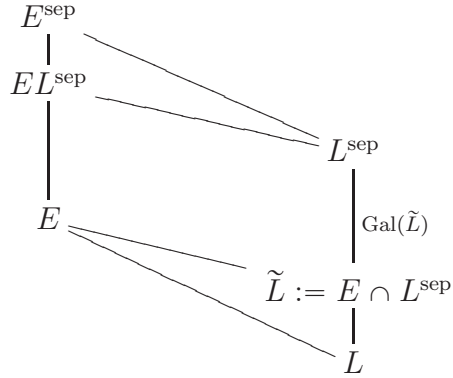
\[
\xymatrix @R=.4cm @C=2cm  {
E^\sep \ar@{-}[d] \ar@{-}[rdd]\\
EL^\sep \ar@{-}[rd] \ar@{-}[dd] & \\
& L^\sep \ar@{-}[dd]^{\Gal(\til L)}\\
E \ar@{-}[rd] \ar@{-}[rdd]\\
& \ \til L:=E \cap L^\sep \ar@{-}[d]\\
& L
}
\]
Here the fields $E$ and $L^\sep$ are linearly disjoint over $\til L$, because $\til L$ is separably closed in $E$; and so the natural map $\Gal(EL^\sep/E) \to \Gal(\til L)$ is an isomorphism.

In the above situation, if a different separable closure of $L$ had been chosen, along with some embedding into $E^\sep$, then the homomorphism $\Gal(E) \to \Gal(L)$ would be modified by conjugation, but 
the injectivity and surjectivity of 
$\Gal(E) \to \Gal(L)$ would not be affected. 

We may apply the lemma in the situation of Notation~\ref{basic notation}, to the field extensions 
$F \to F_P$, $F \to F_U$, $F_P \to F_\wp$, and $F_U \to F_\wp$, where $\wp$ is a branch of $X$ at $P$ lying on $U$.  We then obtain:

\begin{thm} \label{inj Gal maps}
Let $F$ be a semi-global field, and consider field extensions $F\subseteq F_P,F_U\subseteq F_\wp$ as in Notation~\ref{basic notation}. Then
the induced maps $\Gal(F_\wp) \to \Gal(F_P)$ and $\Gal(F_U) \to \Gal(F)$ between absolute Galois groups are injective.  If the branch $\wp$ lies on $U$, then the induced maps $\Gal(F_\wp) \to \Gal(F_U)$ and $\Gal(F_P) \to \Gal(F)$ are injective if and only if the residue field $k$ has
characteristic zero.  The above maps are never surjective.
\end{thm}

\begin{proof}
The last assertion is immediate from Lemma~\ref{Gal map la}, since in each of the corresponding field extensions, the bottom field is not separably closed in the top field.

By Proposition~\ref{branch to point}, every finite separable extension of $F_\wp$ is the compositum of a finite separable extension of $F_P$ with $F_\wp$.  Thus $F_\wp^\sep = F_\wp F_P^\sep$.  So by Lemma~\ref{Gal map la}, 
$\Gal(F_\wp) \to \Gal(F_P)$ is an injection.  Similarly, $\Gal(F_U) \to \Gal(F)$ is injective, using 
Proposition~\ref{component to global}.  If $\cha(k)=0$, then Proposition~\ref{branch to component} implies that
$\Gal(F_\wp) \to \Gal(F_U)$ is injective.  If, in addition, $P$ is a unibranched point of each component of $X$ on which it lies, then Theorem~\ref{point to global} implies that $\Gal(F_P) \to \Gal(F)$ is injective.  

If $\cha(k)=0$ but we do not assume that each of these components is unibranched at $P$, then 
by~\cite[Proposition~6.2]{HHK:H1} there exists a finite Galois split cover $\mc X' \to \mc X$
such that for each closed point $P'\in \mc X'$ lying over $P \in \mc X$, each irreducible component of the closed fiber $X'$ of $\mc X'$ is unibranched at $P'$.  (Recall from~\cite[Section~5]{HHK:H1} that a degree $n$ morphism $\mc X' \to \mc X$ is a {\em split cover} if $\mc X' \times_{\mc X} Q$ consists of $n$ copies of $Q$ for every point $Q \in \mc X$ other than the generic point of $\mc X$.)  
Every split cover is necessarily \'etale; and if we choose a point 
$P' \in \mc X'$ lying over $P$ then the inclusion $F_P \hookrightarrow F'_{P'}$ is an isomorphism, where $F'$ is the function field of $\mc X'$.  Since $F' \subset F'_{P'}$ and since the extension $F'/F$ is algebraic, 
we obtain an inclusion of $F'$ in the algebraic closure of $F$ in $F_P$.  So the map $\Gal(F_P) \to \Gal(F)$ factors through the inclusion $\Gal(F') \hookrightarrow \Gal(F)$.  
The map $\Gal(F_P) \to \Gal(F')$ is injective by Theorem~\ref{point to global} applied to the model~$\mc X'$; and so $\Gal(F_P) \to \Gal(F)$ is also injective.

For the converse, suppose that $\cha k=p>0$.  By 
Proposition~\ref{counterexample prop}(\ref{char p counterex}),
there is 
a degree~$p$ Galois field extension $E_\wp/F_\wp$ that is not induced by any degree $p$ separable field extension of the separable closure $\til F_U$ of $F_U$ in $F_\wp$.  If $E_\wp \subseteq F_\wp^\sep$ is contained in $F_\wp F_U^\sep = F_\wp {\til F_U}^\sep$, then there is a finite Galois extension $E/\til F_U$, say with group $G$, such that $E_\wp \subseteq F_\wp E$.  But $E$ and $F_\wp$ are linearly disjoint over $\til F_U$, since $\til F_U$ is separably closed in $F_\wp$.  So the compositum $F_\wp E$ is a Galois field extension of $F_\wp$ having group $G$.  Let $N = \Gal(F_\wp E/E_\wp)$; this is a normal subgroup of index $p$.  The fixed field $E^N$ is then a degree $p$ separable extension of $\til F_U$ that induces $E_\wp$.  This is a contradiction, showing that actually 
$E_\wp$ is not contained in $F_\wp F_U^\sep$.  Thus $F_\wp^\sep$ is strictly larger than $F_\wp F_U^\sep$, and hence the map $\Gal(F_\wp) \to \Gal(F_U)$ is not injective.  Similarly, using the extension $E_P/F_P$ in  
Proposition~\ref{counterexample prop}(\ref{mixed char counterex}), we deduce that the map $\Gal(F_P) \to \Gal(F)$ is not injective if 
$\cha k \ne 0$.
\end{proof}

\subsection{Van Kampen's theorem for diamonds} \label{van Kampen subsec}

In this section, we prove an analog of van Kampen's theorem in our context.  In the situation of Notation~\ref{geom notn}, the simplest case is the one in which $\mc P, \mc U, \mc B$ each contain just one element.  That is, the closed fiber $X$ of the normal model $\mc X$ is irreducible; $\mc P$ consists of a single unibranched closed point $P$ of $X$; the unique element $U \in \mc U$ is the complement of $P$ in $X$; and the element of $\mc B$ is the unique branch $\wp$ of $X$ at $P$.
Thus $F \subset F_P,\!F_U \subset F_\wp$, and we have a diamond of fields as at the beginning of Section~\ref{descent char 0}.  
In our result, we express the absolute Galois group of our field~$F$ as the amalgamated product of the absolute Galois groups of $F_P,F_U$ over that of $F_\wp$.  This parallels the usual form of van Kampen's theorem in topology, which concerns a space $S = S_1 \cup S_2$ with $S, S_1, S_2$, and $S_0:=S_1 \cap S_2$ connected, and which expresses the fundamental group of $S$ as the amalgamated product of the fundamental groups of $S_1,S_2$ over that of $S_0$.  There, one takes fundamental groups with respect to a common base point in $S_0$; here we will take absolute Galois groups with respect to a chosen separable closure of $F_\wp$ and corresponding separable closures of $F,F_P,F_U$.  See Theorem~\ref{diamond vK}.  Afterwards, in Theorem~\ref{dir lim patches thm} and Theorem~\ref{general vK}, we prove analogous results that consider more general choices of $\mc P, \mc U, \mc B$.
 
For the proof of our analog of van Kampen's theorem, it will be convenient to use the language of torsors.  Let $L$ be a field with separable closure $L^\sep$, and let $G$ be a finite group.  
As discussed at the beginning of Section~\ref{inj subsec}, the group homomorphisms $\phi:\Gal(L) = \Gal(L^\sep/L) \to G$ are in natural bijection with isomorphism classes of pairs consisting of a $G$-Galois \'etale $L$-algebra $E$ together with an $L$-algebra homomorphism $i:E \to L^\sep$.  
On the geometric level, $\Spec(E)$ is a $G$-torsor over $L$, corresponding to the cohomology class $\sigma \in H^1(L,G)$ of the cocycle $\phi \in \Hom(\Gal(L),G) = Z^1(L,G)$.  
(Here $G$ is regarded as a constant finite group scheme over $L$, and so the action of $\Gal(L)$ on $G$ is trivial.)  This torsor over $L$ is {\em geometrically pointed}; i.e., it is equipped with a distinguished $L^\sep$-point, corresponding to the $L$-algebra map $i:E \to L^\sep$.  
A {\em morphism} between two geometrically pointed $G$-torsors over $L$ consists of a morphism of torsors that carries the first base point to the second.  
Thus the isomorphism classes of geometrically pointed $G$-torsors over $L$ are in natural bijection with
$\Hom(\Gal(L),G)$.  Note that there is at most one morphism (necessarily, an isomorphism) between any two geometrically pointed $G$-torsors over $L$,
since a morphism of torsors is determined by the image of a given geometric point.  In particular, a geometrically pointed $G$-torsor has no non-trivial automorphisms.  

\begin{prop} \label{ptd equiv diamond}
Let $F$ be a semi-global field, and let $X$ be the closed fiber of a normal model $\mc X$.  Assume 
that $\mc P, \mc U, \mc B$ as in Notation~\ref{geom notn} each consist of a single element, $P,U,\wp$.
Let $F^\sep,F_P^\sep,F_U^\sep$ be the separable closures of $F,F_P,F_U$ in 
a fixed separable closure $F_\wp^\sep$ of $F_\wp$.  
Let $G$ be a finite group, and consider geometrically pointed $G$-torsors over $F,F_P,F_U, F_\wp$ with respect to the above separable closures.
Then base change induces a bijection between
\renewcommand{\theenumi}{\roman{enumi}}
\renewcommand{\labelenumi}{(\roman{enumi})}
\begin{enumerate}
\item the set of isomorphism classes of geometrically pointed $G$-torsors over $F$; and 
\item the set of pairs consisting of isomorphism classes of geometrically pointed $G$-torsors over $F_P$ and over $F_U$ that induce the same isomorphism class over $F_\wp$.
\end{enumerate}
\end{prop}

This proposition is the key step in proving our van Kampen theorem for diamonds:

\begin{thm} \label{diamond vK}
Let $F$ be a semi-global field, and assume that $\mc P, \mc U, \mc B$ as in Notation~\ref{geom notn} each contain just one element.  
Let $F^\sep,F_P^\sep,F_U^\sep$ be the separable closures of $F,F_P,F_U$ in 
a fixed separable closure $F_\wp^\sep$ of $F_\wp$.  Then 
\[\Gal(F) = \Gal(F_P) *_{\Gal(F_\wp)} \Gal(F_U),\] 
the amalgamated product of $\Gal(F_P)$ with $\Gal(F_U)$ over $\Gal(F_\wp)$.
\end{thm}

\begin{proof}
The theorem asserts that $\Gal(F)$ is
the direct limit of the directed system consisting of the other three groups, with corresponding diagram 
\begin{equation} \label{Gal diamond}
\begin{gathered} 
\xymatrix @R=.7cm @C=.1cm  {
& \Gal(F_\wp) \ar[ld] \ar[rd] &\\
\Gal(F_P) \ar[rd] && \Gal(F_U) \ar[ld]\\
& \Gal(F) & .
}
\end{gathered}
\end{equation} 
That is, for every finite group $G$, the natural map of sets
\[\Hom(\Gal(F),G) \to \Hom(\Gal(F_P),G) \times_{\Hom(\Gal(F_\wp),G)} \Hom(\Gal(F_U),G)\]
is a bijection.  But as noted above, $\Hom(\Gal(F),G)$ is in natural bijection with the set of isomorphism classes of geometrically pointed $G$-torsors over $F$, and similarly for $F_P, F_U, F_\wp$.  So the assertion follows from Proposition~\ref{ptd equiv diamond}. 
\end{proof}

\begin{remark}
\begin{enumerate}
\renewcommand{\theenumi}{\alph{enumi}}
\renewcommand{\labelenumi}{(\alph{enumi})}  
\item
If one considers schemes rather than function fields, then 
the analog of Theorem~\ref{diamond vK} holds but not that of Theorem~\ref{inj Gal maps}.  More precisely, take a nonempty finite set of closed points of the generic fiber
of $\mc X$; let $\Sigma$ be its closure in $\mc X$; and let $\Sigma_P, \Sigma_U, \Sigma_\wp$ be the pullbacks of $\Sigma$ from $\mc X$ to $\Spec(\wh R_P), \Spec(\wh R_U), \Spec(\wh R_\wp)$ respectively.  
Replacing $\Gal(F),\Gal(F_P),\Gal(F_U),\Gal(F_\wp)$ in Theorem~\ref{diamond vK} by the fundamental groups of 
$\mc X\smallsetminus{}\Sigma$, $\Spec(\wh R_P)\smallsetminus{}\Sigma_P$, $\Spec(\wh R_U)\smallsetminus{}\Sigma_U$,
$\Spec(\wh R_{\wp})\smallsetminus{}\Sigma_{\wp}$, the analog of Theorem~\ref{diamond vK} holds, by using 
formal patching (e.g.\ \cite[Theorem~3.2.8]{Har:MSRI}) instead of patching over fields.  
But the analog of Theorem~\ref{inj Gal maps} fails if we let 
$\mc X = \mbb P^1_{k[[t]]}$, $U = \mbb A^1_k$, and $\Sigma = (x=0)$, with $\cha(k)=0$.
Namely, $y^n=x$ defines a branched cover of $\Spec(\wh R_U)$ unramified away from $\Sigma_U$, but it is not induced by a branched cover of $\mc X\smallsetminus{}\Sigma$.  Thus, as in Lemma~\ref{Gal map la}, the homomorphism $\pi_1(\Spec(\wh R_U)\smallsetminus{}\Sigma_U) \to
\pi_1(\mc X\smallsetminus{}\Sigma)$ is not injective.  Similarly, if $P$ is the point $x=0$ on the closed fiber, the map $\pi_1(\Spec(\wh R_P)\smallsetminus{}\Sigma_P) \to\pi_1(\mc X\smallsetminus{}\Sigma)$ is not injective.

\item 
The absolute Galois groups $\Gal(F)$, $\Gal(F_P)$, and $\Gal(F_U)$
that arise in Theorem~\ref{diamond vK} have interesting structures.  In particular,  
every finite group is a quotient of each of these profinite groups.  In the case of the group $\Gal(F)$, this was shown for $F=K(x)$ in \cite[Corollary~2.4]{Ha:GCAL}; for a more general semi-global field $F$, the assertion is a special case of \cite[Theorem~2.7]{Pop:SC}.  For the groups $\Gal(F_P)$ and $\Gal(F_U)$, the assertion is given by \cite[Corollary~3.20]{Lef}.  On the other hand, none of these absolute Galois groups are free; this is because the fields $F ,F_P, F_U$ each have cohomological dimension greater than one, and so they are not even projective (see \cite[Proposition~11.6.6, Corollary~22.4.6]{FJ}).  The group $\Gal(F_\wp)$ is also not free, because $F_\wp$ is a complete discretely valued field.  In general, not every finite group is a quotient of $\Gal(F_\wp)$.  For example, if $K=\mbb C((t))$ and $F=K(x)$, with $P$ being the point $x=0$ on the closed fiber $\mbb P^1_\mbb C$ of $\mbb P^1_K$, then $F_\wp$ is isomorphic to $\mbb C((x))((t))$, and so the finite quotients of $\Gal(F_\wp)$ are just the metacyclic groups. 
\end{enumerate}
\end{remark}

Proposition~\ref{ptd equiv diamond} follows from patching of $G$-torsors (or equivalently, of $G$-Galois \'etale algebras; see \cite[Theorem~7.1]{HH:FP}) combined with the fact that a geometrically pointed $G$-torsor has no non-trivial automorphisms.  Below we prove a more general result, 
Proposition~\ref{pointed patching}, 
which will be used in obtaining variants of Theorem~\ref{diamond vK} in the next section, and which concerns torsors that are equipped with a family of geometric points, rather than just one such point.   
More precisely, let $L$ be a field and $G$ a finite group. Let $\ms S = \{L_s\}_{s \in S}$ be a nonempty indexed set of
separable closures $L_s$ of $L$. 
Define an $\ms S$-{\em multipointed $G$-torsor}
over $L$ to be a pair $(Z, \{Q_s\}_{s \in S})$, where $Z$ is a $G$-torsor
over $L$, and $Q_s: \Spec(L_s) \to Z$ is an $L$-morphism for each $s \in S$, i.e., an $L_s$-point of $Z$.  A {\em morphism} of $\ms S$-multipointed $G$-torsors is a morphism of the underlying torsors that carries the chosen geometric points of the first torsor to the corresponding points of the second.  Write $\ms T_G^{\ms S}(L)$ for the category of $\ms S$-multipointed $G$-torsors over $L$.  Since the indexed set $\ms S$ is nonempty, the objects in this category have no non-trivial automorphisms, and between any two objects there is at most one morphism (necessarily an isomorphism); in the terminology of \cite[Tag 02XZ]{stacks-project}, 
one says that $\ms T_G^{\ms S}(L)$ is a {\em setoid}.

Let $F$ be a semi-global field, and let $\mc P, \mc U, \mc B$ be as in Notation~\ref{geom notn}.
For each branch $\wp
\in \mc B$, choose a separable closure $F^\sep_\wp$ of $F_\wp$, and let $\ms S_\wp$
denote the singleton set consisting of $F^\sep_\wp$. If a branch $\wp$ lies on $U
\in \mc U$ at $P \in \mc P$, then we let $F^\sep_U(\wp)$ (respectively $F^\sep_P(\wp)$) denote the separable closure of $F_U$ (resp.\ $F_P$) in
$F^\sep_\wp$.
Let $\ms S_U$ (respectively $\ms S_P$) denote the indexed collection of fields of the form $F_U^\sep(\wp)$ (resp.~$F^\sep_P(\wp)$), where $\wp$ ranges over the branches on $U$ (resp.\ at $P$). 
Finally, let $F^\sep(\wp)$ denote the separable closure of $F$ in $F^\sep_\wp$ and let $\ms S$ denote the indexed collection of fields $\{F^\sep(\wp)\}_{\wp \in \mc B}$.
We define product categories
\[
\ms T_G^{\mc U} = \prod_{U \in \mc U} \ms T_G^{\ms S_U}(F_U), \ \
\ms T_G^{\mc P} = \prod_{P \in \mc P} \ms T_G^{\ms S_P}(F_P), \ \
\ms T_G^{\mc B} = \prod_{\wp \in \mc B} \ms T_G^{\ms S_\wp}(F_\wp). \]

The natural inclusions of fields $F^\sep(\wp) \subset F^\sep_U(\wp),\!F^\sep_P(\wp) \subset F^\sep_\wp$ induce functors
\[\ms T_G^{\ms S}(F) \longrightarrow \ms T_G^{\mc U},\!\ms T_G^{\mc P} \longrightarrow \ms T_G^{\mc B}.\]

Recall that if $\alpha_i:\mc C_i\rightarrow \mc C_0$ are functors ($i=1,2$), the 2-fiber product category $\mc C_1\times_{\mc C_0}\mc C_2$ (with respect to $\alpha_1$, $\alpha_2$) is the category whose objects are triples $(V_1,V_2,\phi)$ consisting of objects $V_i\in \mc C_i$ and an isomorphism $\phi: \alpha_1(V_1)\rightarrow \alpha_2(V_2)$ in $\mc C_0$. A morphism from $(V_1,V_2,\phi)$ to $(V_1',V_2',\phi')$ is a pair of morphisms $f_i:V_i\rightarrow V_i'$ in $\mc C_i$ ($i=1,2$) such that $\phi'\circ \alpha_1(f_1)=\alpha_2(f_2)\circ \phi$.

\begin{prop} \label{pointed patching}
The above functors induce an equivalence of categories
\[\ms T_G^{\ms S}(F) \to \ms T_G^{\mc U} \times_{\ms T_G^{\mc B}} \ms T_G^{\mc P}, \]
where the right hand side is a 2-fiber product of categories.
The isomorphism classes of objects in the category on the right hand side are in natural bijection with the (1-)fiber product of isomorphism classes of objects in the respective categories.
\end{prop}

Note that Proposition~\ref{ptd equiv diamond} is a special case of the second assertion in this proposition. 

\begin{proof}
The second assertion follows from the first since all the categories involved are setoids.  For the first assertion there are two steps.

	\smallskip

	{\em Step 1:} Essential surjectivity.

	An object of the right hand side of the map in the statement of the theorem corresponds to the following data: $G$-torsors $Z_P = \Spec(E_P), Z_U = \Spec(E_U), Z_\wp = \Spec(E_\wp)$ for all $P \in \mc P$, $U \in \mc U$, $\wp \in \mc B$; together with associated points $\zeta_\wp \in Z_\wp(F_\wp^\sep)$ and $\zeta_\xi(\wp) \in Z_\xi(F_\xi^\sep(\wp))$ for $\wp$ a point at (or on) $\xi \in \mc P \cup \mc U$;
	such that for each pair $\xi,\wp$ as above there exists an isomorphism (necessarily unique) of $G$-torsors $(Z_\xi)_{F_\wp} \to Z_\wp$ that takes $\zeta_\xi(\wp)$ to $\zeta_\wp$. Thus we obtain a patching problem of $G$-torsors, or equivalently of $G$-Galois algebras, which has a solution that is unique up to isomorphism (by \cite[Theorem~7.1]{HH:FP}); viz., a $G$-torsor $Z = \Spec(E)$ over $F$ that induces each of the torsors $Z_P$, $Z_U$, $Z_\wp$ compatibly.
	If $\wp$ is a branch at $P$ on $U$,
	the points $\zeta_{P,\wp},\zeta_{U,\wp},\zeta_\wp$ correspond to $F$-homomorphisms
	$E_P \to F_P^\sep(\wp) \subset F_\wp^\sep, E_U \to F_U^\sep(\wp) \subset F_\wp^\sep, E_\wp \to F_\wp^\sep$ such that the first two are restrictions of the third.  These homomorphisms thus restrict to a
	common $F$-homomorphism $i_\wp:E \to F_\wp^\sep$.  Since $E$ is a finite \'etale $F$-algebra,
	the image of $i_\wp$ lies in $F^\sep(\wp)$, the separable closure of $F$ in $F_\wp^\sep$.
	It follows that
	the pair $(Z,\{i_\wp\}_{\wp \in \mc B})$ maps to the isomorphism class of our initially chosen~object.

	\smallskip

	{\em Step 2:} Full faithfulness.

	Since the categories in question are setoids, we need only check that objects that become isomorphic under our functor were isomorphic to start with.

	Consider two objects
	$(Z,\{i_s\}_{s \in S})$, $(Z',\{i_s'\}_{s \in S})$ from the left hand side.
	Since they have isomorphic images, the induced objects on the right hand side are isomorphic
	as multipointed $G$-torsors;
	i.e., for each $\xi \in
	\mc P \cup \mc U \cup \mc B$ there is a (unique) torsor isomorphism $j_\xi:Z_{F_\xi} \to Z'_{F_\xi}$ that carry the base points of each $Z_{F_\xi}$ to the base points of
	$Z'_{F_\xi}$.  The $j_\xi$ are compatible, by uniqueness.  Hence they
define an isomorphism of $G$-torsor patching problems; and by \cite[Theorem~7.1]{HH:FP}, this isomorphism is induced by a unique $G$-torsor isomorphism $j:Z \to Z'$.  Necessarily, $j$ takes the base points of $Z$ to those of $Z'$, since $j_\xi$ takes the base points of $Z_{F_\xi}$ to those of
	$Z'_{F_\xi}$, and since in each case the geometric points are in bijection with $G$.  So $(Z,\{i_s\}_{s \in S})$, $(Z',\{i_s'\}_{s \in S})$ are isomorphic as multipointed $G$-torsors.
\end{proof}

\subsection{Van Kampen's theorem for general reduction graphs} \label{genl van Kampen subsec}

In this section we prove variants on Theorem~\ref{diamond vK} in which the sets $\mc P, \mc U, \mc B$ in Notation~\ref{geom notn} can consist of more than one element, and so the configuration of fields need not be a diamond.  The simplest generalization would assert that $\Gal(F)$ is the direct limit of the system of absolute Galois groups $\Gal(F_\xi)$ for $\xi \in \mc P \cup \mc U \cup \mc B$, with respect to homomorphisms $\Gal(F_\wp) \to \Gal(F_P)$ and $\Gal(F_\wp) \to \Gal(F_U)$ whenever $\wp$ is a branch at $P$ on $U$.  Here the absolute Galois groups are taken with respect to a suitable choice of separable closures $F^\sep$ of $F$ and $F_\xi^\sep$ of $F_\xi$ for $\xi \in \mc P \cup \mc U \cup \mc B$.  For $\Gal(F)$ to be the direct limit, we would need homomorphisms $\Gal(F_\xi) \to \Gal(F)$ for all $\xi \in \mc P \cup \mc U \cup \mc B$ such that the compositions $\Gal(F_\wp) \to \Gal(F_P) \to \Gal(F)$ and $\Gal(F_\wp) \to \Gal(F_U) \to \Gal(F)$ agree.

By the discussion at the 
beginning of Section~\ref{Gal gps sec}, such group homomorphisms would be induced by  
a choice of separable closures $F^\sep, F_\xi^\sep, F_\wp^\sep$ for $\xi \in \mc P \cup \mc U$, $\wp \in \mc B$, such that $F^\sep$ is the separable closure of $F$ in $F_\xi^\sep$ for every $\xi \in \mc P \cup \mc U \cup \mc B$, and $F_\xi^\sep$ is the separable closure of $F_\xi$ in $F_\wp^\sep$ for every branch $\wp \in \mc B$ on (or at) $\xi \in \mc P \cup \mc U$.  We call this a {\em compatible system of separable closures}; and it then makes sense to ask whether the van Kampen assertion holds.

The next result provides a necessary and sufficient condition for such a generalized van Kampen theorem to hold, in terms of the reduction graph associated to the sets $\mc P, \mc U, \mc B$.  
(As in \cite[Section~2.1.1]{HHK:Hi}, the {\em reduction graph} associated to these sets is the connected bipartite graph whose vertices are the elements of $\mc P \cup \mc U$ and whose edges are the elements of $\mc B$, where an edge $\wp$ connects two vertices $P,U$ if $\wp$ is a branch at $P$ lying on $U$.)

\begin{thm}  \label{dir lim patches thm}
Let ${\mc X}$ be a normal model for a semi-global field $F$, and let $\mc P,\mc U, \mc B$ be as in Notation~\ref{geom notn}. 
Then the following conditions are equivalent: 
\begin{enumerate}
\renewcommand{\theenumi}{\roman{enumi}}
\renewcommand{\labelenumi}{(\roman{enumi})}
\item \label{tree item}
The associated reduction graph is a tree.
\item \label{compat system item}
There is a compatible system of separable closures of the fields $F$ and $F_\xi$ for 
$\xi \in \mc P,\mc U, \mc B$.
\item \label{dir lim genl item}
$\Gal(F)$ is the direct limit of the groups $\Gal(F_\xi)$ for $\xi \in \mc P \cup \mc U \cup \mc B$ with respect to some compatible system of separable closures.
\end{enumerate}
Under these equivalent conditions, 
$\Gal(F)$ is the direct limit of the groups $\Gal(F_\xi)$ for $\xi \in \mc P \cup \mc U \cup \mc B$ with respect to \textbf{any} given compatible system of separable closures.
\end{thm}

\begin{proof}
Throughout this proof, the letter $\Gamma$ will denote the associated reduction graph.
To show that
(\ref{tree item}) implies (\ref{compat system item}), we construct these separable closures inductively, using that the reduction graph $\Gamma$ is a tree.  
Namely, consider a subtree $\mc T$ of $\Gamma$, let $\xi_0 \in \mc P \cup \mc U$ be a terminal vertex of $\mc T$, and let $\wp_0 \in \mc B$ be the edge connecting $\xi_0$ to the rest of $\mc T$, which we call $\mc T'$.  Suppose that we have compatible separable closures $F_\xi^\sep$ associated to the vertices and edges $\xi$ of $\mc T'$.  Let $\xi_1 \in \mc P \cup \mc U$ be the other vertex of the edge $\wp_0$; this is a vertex of $\mc T'$.  Choose a separable closure $F_{\wp_0}^\sep$ of $F_{\wp_0}$ that contains 
$F_{\xi_1}^\sep$.  (For example, take any separable closure of the field $(F_{\wp_0} \otimes_{F_{\xi_1}} F_{\xi_1}^\sep)/{\frak m}$, where $\frak m$ is a maximal ideal.)  Then
take $F_{\xi_0}^\sep$ to be the separable closure of $F_{\xi_0}$ in $F_{\wp_0}^\sep$.  
Since $F_{\xi_1}^\sep$ is the separable closure of $F_{\xi_1}$ in $F_{\wp_0}^\sep$,
it follows that the separable closure of $F$ in $F_{\wp_0}^\sep$ is the same as the separable closure of $F$ in $F_{\xi_1}^\sep$.  But the latter is the same as the separable closure of $F$ in $F_\wp^\sep$, where $\wp \in \mc B$ is any branch at (or on) $\xi_1$, and that field is $F^\sep$.  So this system of separable closures on the vertices and edges of $\mc T$ is compatible, thus completing the induction, and showing that (\ref{tree item}) implies (\ref{compat system item}).

Next, we show that with respect to any given compatible system of separable closures $F^\sep$ and $F_\xi^\sep$, $\Gal(F)$ is the direct limit of the groups $\Gal(F_\xi)$.  This will show that  
(\ref{compat system item}) implies both (\ref{dir lim genl item}) and the stronger condition in the last part of the assertion.  For this, note that given this compatible system, the construction preceding Proposition~\ref{pointed patching} yields indexed sets $\ms S_P, \ms S_U, \ms S$ of separable closures of $F_P, F_U, F$ respectively, for $P \in \mc P$ and $U \in \mc U$.  
For any given $\xi \in \mc P \cup \mc U$, the fields $F_\xi^\sep(\wp)$ in the indexed set $\ms S_\xi$ are each just the field $F_\xi^\sep$ of the previous paragraph; and similarly all of the fields in the indexed set $\ms S$ are just the above field $F^\sep$.  
A geometrically pointed $G$-torsor over $F$ (with respect to $F^\sep$) is the same as an $\ms S$-pointed $G$-torsor $(Z,\{Q_s\}_{s \in S})$ such that the morphisms $Q_s:\Spec(F^\sep(\wp))\to Z$ are all the same morphism $\Spec(F^\sep) \to Z$.  The corresponding statements hold for each $F_P, F_U, F_\wp$. 
Let $\bar{\ms T}_G^{\ms S}(F)$ denote the full subcategory of $\ms T_G^{\ms S}(F)$ whose objects are the geometrically pointed $G$-torsors over $F$ under the above identification. Similarly, we define full subcategories $\bar{\ms T}_G^{\mc U}$, $\bar{\ms T}_G^{\mc P}$, and $\bar{\ms T}_G^{\mc B}$ of the product categories ${\ms T}_G^{\mc U}$, ${\ms T}_G^{\mc P}$, and ${\ms T}_G^{\mc B}$ defined in the discussion leading up to Proposition~\ref{pointed patching}.  Under the equivalence of categories given in Proposition~\ref{pointed patching}, an object in $\ms T_G^{\ms S}(F)$ is sent to an object in
$\bar{\ms T}_G^{\mc U} \times_{\bar{\ms T}_G^{\mc B}} \bar{\ms T}_G^{\mc P}$
if and only if it is an object in $\bar{\ms T}_G^{\ms S}(F)$, since the reduction graph is connected.  Thus the equivalence of categories in Proposition~\ref{pointed patching} restricts to an equivalence $\bar{\ms T}_G^{\ms S}(F) \to \bar{\ms T}_G^{\mc U} \times_{\bar{\ms T}_G^{\mc B}} \bar{\ms T}_G^{\mc P}$.  As in Proposition~\ref{pointed patching}, isomorphism classes of objects in
$\bar{\ms T}_G^{\mc U} \times_{\bar{\ms T}_G^{\mc B}} \bar{\ms T}_G^{\mc P}$ are in natural bijection with the (1-)fiber product of isomorphism classes of objects in the respective categories.  So the desired assertion now follows from the natural bijection between $\Hom(\Gal(F),G)$ and isomorphism classes of geometrically pointed $G$-torsors.  

It remains to show that (\ref{dir lim genl item}) implies (\ref{tree item}). 
Suppose to the contrary that the reduction graph $\Gamma$ is not a tree.  We claim that that there is a non-trivial finite Galois field extension $E/F$ with Galois group~$G$ that induces the trivial extension over each $F_P$ and each $F_U$, for $P \in \mc P$ and $U \in \mc U$.  Once this is shown, the corresponding non-trivial map $\Gal(F) \to G$ induces the trivial maps $\Gal(F_\xi) \to G$ for $\xi \in \mc P \cup \mc U$.  But those trivial maps are also induced by the trivial map $\Gal(F) \to G$.  So $\Gal(F)$ does not have the universal property for direct limits.  This shows that it suffices to prove the claim.

To do this, first assume that the set $\mc P$ contains all the points where the closed fiber $X$ is not unibranched.  By~\cite[Proposition~6.2]{HHK:H1}, since $\Gamma$ is not a tree (and 
thus has a non-trivial covering space), there exists a non-trivial 
finite connected split cover $\mc Y \to \mc X$.  Let $E/F$ be the 
corresponding finite separable field extension. By~\cite[Corollary~5.5]{HHK:H1}, this split cover induces trivial extensions of each $F_\xi$, for $\xi \in \mc P \cup \mc U$; hence $E/F$  satisfies the conditions of the claim.

If we do not make the above assumption on the set $\mc P$, then the proof of~\cite[Proposition~6.2]{HHK:H1}
still shows that the finite connected covering spaces of $\Gamma$ are in bijection with the split covers of $\mc X$ that induce trivial extensions of each $F_\xi$.  So again, since $\Gamma$ has non-trivial covering spaces, there exists a non-trivial finite connected split cover $\mc Y \to \mc X$ that is trivial over each $F_\xi$; and the corresponding field extension $E/F$ again satisfies the conditions of the claim.
\end{proof}

Thus, in the context of Notation~\ref{geom notn}, the absolute Galois group of $F$ is the direct limit of the absolute Galois groups of the fields $F_P,F_U,F_\wp$ with respect to a compatible system of separable closures if the reduction graph is a tree, but not otherwise.  
We now consider the case where the reduction graph is not a tree, and state a van Kampen theorem in terms of groupoids rather than groups, to avoid this limitation.  This assertion parallels a result in topology that generalizes the usual van Kampen theorem to the case in which the intersection is allowed to be disconnected, doing so in terms of groupoids (see \cite{Brown}, Section~6.7).  Our approach here is also motivated by the use of fundamental groupoids of schemes in \cite[V.7]{SGA1}; there, as in the topological context, one uses a collection of base points, rather than just one point.  In this way we can avoid the problem that in general there is no
compatible system of separable closures (cf.~Theorem~\ref{dir lim patches thm}).

Recall that a {\em groupoid} is a category in which every homomorphism is an isomorphism.  Groups can be viewed as groupoids, by associating to each group $G$ the groupoid $BG$ consisting of one object, and with the morphisms corresponding to the elements of $G$.  If $L$ is a field, and $\ms S = \{L_s\}_{s \in S}$ is a nonempty indexed set of separable closures of $L$, then we may consider the 
{\em absolute Galois groupoid} $\pi_1(L,\ms S)$ of $L$ with respect to $\ms S$.  Its objects are the elements of $S$, and its morphisms are isomorphisms between the corresponding separable closures of $L$.  
(Here $\pi_1(L,\ms S)$ can be viewed as the fundamental groupoid $\pi_1(\Spec(L),\ms S)$ of $\Spec(L)$ with respect to the indexed set of base points $\ms S$; cf.\ \cite[V.7]{SGA1} and \cite[Section~6.7]{Brown}.)
This groupoid is small (i.e., is a small category) since $S$ is a set.

If $L^\sep$ is a separable closure of a field $L$, and we write 
$\Gal(L) = \Gal(L^\sep/L)$, then for each finite group $G$ we have 
a natural bijection between the objects of 
$\Hom(B\!\Gal(L),BG)$ and the elements of $\Hom(\Gal(L),G)$.  
As discussed at the beginning of Section~\ref{van Kampen subsec}, 
this latter set is in natural bijection with the set of geometrically pointed $G$-torsors over $L$; i.e., the set of 
isomorphism classes of pairs $(Z,i)$, where $Z = \Spec(E)$ is a $G$-torsor over $L$ and $i:E \to L^\sep$ is an $L$-algebra map, and where $i$ corresponds to a choice of a distinguished   
$L^\sep$-point on $Z$.  This bijection can be extended to the context of groupoids, 
using multipointed torsors 
(with notation as in the discussion leading up to Proposition~\ref{pointed patching}):

\begin{lemma} \label{pointed torsors}
Let $L$ be a field, $\ms S = \{L_s\}_{s \in S}$ an indexed set of separable closures of $L$,
and $G$ a finite group.  
Then the set $\Hom(\pi_1(L,\ms S),BG)$ is in natural
bijection with the set of isomorphism classes of $\ms S$-multipointed $G$-torsors
$(Z,\{i_s\}_{s \in S})$ in $\ms T_G^{\ms S}(L)$.
\end{lemma}

\begin{proof}
First, consider the isomorphism class of $(Z,\{i_s\}_{s \in S})$ as above, where $Z = \Spec(E)$,
For each $s \in S$, let $\zeta_s \in Z(L_s)$ be the $L_s$-point corresponding to $i_s: E
\to L_s$.  
For $s \in S$, the restricted multipointed torsor $(Z,\{i_s\})$ defines an element
$f_s \in \Hom(\Gal(L_s/L),G) = \Hom(\pi_1(L,\{L_s\}),BG)$.  Given $s,s' \in S$,
each $L$-algebra isomorphism $\alpha:L_{s'} \to L_s$ induces
a bijection $Z(\alpha):Z(L_{s'}) \to Z(L_s)$;
and there is a unique $g_\alpha \in G$ such that $Z(\alpha)(\zeta_{s'}) = \zeta_s \cdot g_\alpha$.
Note that $g_\alpha = f_s(\alpha)$ if $s'=s$.
It is then straightforward to check that there is a morphism $f \in \Hom(\pi_1(L,\ms S),BG)$ given by $f(\alpha)=g_\alpha$ as above for all $s,s'$ and $\alpha:L_{s'} \to L_s$; and that for $s \in S$, the restriction of $f$ to $\Gal(L_s/L)$ is $f_s$.
This defines one direction of the bijection.

For the opposite direction, we begin by
picking some $s_0 \in S$; and for every $s \in S$ we pick an $L$-algebra isomorphism $\alpha_s:L_{s_0} \to L_s$, with $\alpha_{s_0}$ being the identity automorphism of $L_{s_0}$.  
This induces a conjugation map $c_{\alpha_s}:\Gal(L_{s_0}/L) \to \Gal(L_s/L)$,
sending $\sigma$ to $\alpha_s \sigma \alpha_s^{-1}$.
Say $f \in \Hom(\pi_1(L,\ms S),BG)$.  Then for every $s \in S$, $f$ restricts to an element $f_s \in \Hom(\pi_1(L,\{L_s\}),BG) = \Hom(\Gal(L_s/L),G)$, corresponding to the isomorphism class of a $G$-torsor $Z_s = \Spec(E_s)$ over $L$ together with
an $L$-homomorphism $i_s:E_s \to L_s$; here $i_s$ corresponds to
an $L_s$-point $\zeta_s$ on $Z_s$.
Write $Z=Z_{s_0}$, $E=E_{s_0}$, and $\zeta = \zeta_{s_0}$, and for each $s$ let
$Z(\alpha_s):Z(L_{s_0}) \to Z(L_s)$ be the map induced by $\alpha_s$.
Since $f$ is a morphism,
the two maps $f_{s_0},f_s c_{\alpha_s} \in \Hom(\pi_1(L,\{L_{s_0}\}),BG) = \Hom(\Gal(L_{s_0}/L),G)$
differ by conjugation by $f(\alpha_s) \in G$.  So there is a unique isomorphism $\alpha_{s*}:Z \to Z_s$ of $G$-torsors over $L$ that carries $Z(\alpha_s)(\zeta) \in Z(L_s)$ to $\zeta_s \cdot f(\alpha_s) \in Z_s(L_s)$.
We then obtain a multipointed torsor $(Z,\{i_s\}_{s \in S})$,
where $i_s:E \to L_s$ is the homomorphism corresponding to the $L_s$-point
$\alpha_{s*}^{-1}(\zeta_s) =
Z(\alpha_s)(\zeta) \cdot f(\alpha_s)^{-1}$ on $Z$.
In this way, for each $f$ we obtain the isomorphism class of a multipointed
torsor $(Z,\{i_s\}_{s \in S})$.
It is straightforward to check that this association
is independent of the choices of $s_0$ and $\alpha_s$, and is
inverse to the one in the previous paragraph.
\end{proof}

Note that since there is at most one
morphism between any two objects of $\ms T_G^{\ms S}(L)$,
the above lemma yields an equivalence of categories $\Hom(\pi_1(L,\ms S),BG) \cong \ms T_G^{\ms S}(L)$, 
if we regard the set $\Hom(\pi_1(L,\ms S),BG)$ as a category
with all arrows being identities.

In the context of the discussion leading up to
Proposition~\ref{pointed patching}, we may take the disjoint union groupoid
$\coprod_{P \in \mc P} \pi_1(F_P,\ms S_P)$, whose objects and morphisms are the disjoint unions of the objects and morphisms of the groupoids $\pi_1(F_P,\ms S_P)$, for $P \in \mc P$.  Similarly we may take 
$\coprod_{U \in \mc U} \pi_1(F_U,\ms S_U)$ and $\coprod_{\wp \in \mc B} \pi_1(F_\wp,\ms S_\wp)$.  We then 
have a commutative diagram of groupoids, which generalizes diagram~(\ref{Gal diamond}) in Section~\ref{van Kampen subsec}, and in which the arrows induce bijections on the (finite) sets of objects of the four categories:
\begin{equation} \label{groupoid diamond diag}
\begin{gathered}
\xymatrix @R=.7cm @C=.1cm  {
& \coprod \pi_1(F_\wp,\ms S_\wp) \ar[ld] \ar[rd] &\\
\coprod \pi_1(F_P,\ms S_P) \ar[rd] && \coprod \pi_1(F_U,\ms S_U) \ar[ld]\\
& \pi_1(F,\ms S) &
}
\end{gathered}
\end{equation}
Here, the commutativity assertion is that the two vertical compositions give the same (not just equivalent) maps on objects, and on morphisms.

We now obtain a van Kampen-type theorem in terms of groupoids, which generalizes Theorem~\ref{diamond vK}, and
parallels the topological van Kampen result \cite[6.7.2]{Brown} for groupoids:

\begin{thm} \label{general vK}
The above diamond is a pushout diagram of small groupoids, in the sense that for every small groupoid $\mc G$,
the natural map of sets
\[\Hom(\pi_1(F,\ms S),\mc G) \to  \Hom(\coprod \pi_1(F_P,\ms S_P),\mc G) \times_{\Hom(\coprod \pi_1(F_\wp,\ms S_\wp),\mc G)}  \Hom(\coprod \pi_1(F_U,\ms S_U),\mc G)\]
is a bijection.  For any element of $S$, corresponding to a separable closure $F^\sep$ of $F$, the absolute Galois group $\Gal(F^\sep/F)$ of $F$ is the 
automorphism group of that object in this groupoid.
\end{thm}

\begin{proof} 
The last assertion is immediate from the main assertion.  The proof of the main assertion 
begins with several reduction steps.

First note that the category $\pi_1(F,\ms S)$ is connected (viz.\ there is a morphism between each pair of objects in this category), because any two separable closures of a field $L$ are $L$-isomorphic.  So we may assume that $\mc G$ is connected, by treating each connected component separately.

Second, we reduce to the case that $\mc G$ has just one object, i.e., it is of the form $BG$ for some group $G$.  Pick an object $t_0$ in $\mc G$, and for every object $t$ in $\mc G$ pick an isomorphism $j_t:t_0 \to t$, with $j_{t_0}$ being the identity on $t_0$.  Let $G = \Aut(t_0)$, so that $BG$ is the full subcategory of $\mc G$ whose unique object is $t_0$.  Define a functor 
$J:\mc G \to BG$ by taking every object in $\mc G$ to the object $t_0$ of $BG$, and taking every morphism 
$\alpha \in \Hom(t,t')$ in $\mc G$ to the morphism $j_{t'}^{-1}\alpha j_t \in \End(t_0)$ in $BG$.
Now given an element $(\phi_{\mc P},\phi_{\mc U})$ in the right hand side of the above map of sets, by composing with $J$ we obtain an element $(J\phi_{\mc P},J\phi_{\mc U})$ in
\[\Hom(\coprod \pi_1(F_P,\ms S_P),BG) \times_{\Hom(\coprod \pi_1(F_\wp,\ms S_\wp),BG)} \Hom(\coprod \pi_1(F_U,\ms S_U),BG).\]
Once we prove the result for maps to groupoids that have just one object, we have that there is a unique $\til \phi \in \Hom(\pi_1(F,\ms S),BG)$ that induces $(J\phi_{\mc P},J\phi_{\mc U})$.  
Define the functor
$\phi:\pi_1(F,\ms S)\to\mc G$ by taking each object
$\wp \in \mc B = \Obj(\pi_1(F,\ms S))$ to
$\phi_{\mc P}(\wp) = \phi_{\mc U}(\wp)$, for $\wp$ a branch at $P$ on $U$ (so that $\wp$ is also an object in $\pi_1(F_P,\ms S_P)$
and in $\pi_1(F_U,\ms S_U)$, which allows us to apply $\phi_{\mc P}$ and $\phi_{\mc U}$); and taking each morphism $\alpha:\wp \to \wp'$ (i.e., each $F$-algebra isomorphism $\alpha:F^\sep(\wp) \to F^\sep(\wp')$) to
$j_{\phi(\wp')} \til \phi(\alpha) j_{\phi(\wp)}^{-1}:
\phi(\wp)\to\phi(\wp')$.
Then $\phi$ is the unique element in $\Hom(\pi_1(F,\ms S),\mc G)$ that maps to 
$(\phi_{\mc P},\phi_{\mc U})$.  This establishes the desired bijection and completes this reduction step.

Third, since the set of objects in each of the groupoids in the diamond is finite, and since the automorphism group of each object is profinite, it suffices to prove the result in the case that $\mc G = BG$ for $G$ a finite group.  We now assume that we are in that case.

By Lemma~\ref{pointed torsors}, 
we may identify
\[\Hom(\coprod \pi_1(F_P,\ms S_P),BG) = \prod \Hom(\pi_1(F_P, \ms S_P),BG)\]
with the set of isomorphism classes of objects in $\ms T_G^{\mc P}$, and similarly for $\mc B, \mc U$; and we may identify $\Hom(\pi_1(F, \ms S), BG)$ with the set of isomomorphism classes in $\ms T_G^{\ms S}(F)$. The result therefore follows from Proposition~\ref{pointed patching}.
\end{proof}

Following \cite{Stix}, we can also describe the absolute Galois group $\Gal(F^\sep/F)$ of $F$ more explicitly, by making a choice of maximal tree $\mc T$ in the reduction graph $\Gamma$ of $(\mc X,\mc P)$.  The vertices of $\mc T$ are the same as those of $\Gamma$, and are indexed by $\mc P \cup \mc U$.  For any two vertices $v_1,v_2$ of $\Gamma$, there is a unique minimal path in $\mc T$ from $v_1$ to $v_2$, and this provides an isomorphism between the fundamental groups $\pi_1(\Gamma,v_i)$ for $i=1,2$.  These groups can also be identified with the fundamental group of $\Gamma$ with respect to $\mc T$ as a ``base point''; or equivalently, the fundamental group of the graph $\Gamma/\mc T$ obtained from $\Gamma$ by contracting $\mc T$ to a single vertex.  The graph $\Gamma/\mc T$ has just one vertex, and its edges are in bijection with the edges of $\Gamma$ that do not lie in $\mc T$.  This fundamental group is thus free of finite rank, with generators $e_\wp$ indexed by those branches $\wp \in \mc B$ that correspond to the edges of $\Gamma/\mc T$.  In this situation, Corollary~3.3 of \cite{Stix} gives:

\begin{prop} \label{free product quotient}
Let $F$ be a semi-global field, and let $\mc P$, $\mc U$, $\mc B$ be as in Notation~\ref{geom notn}. For each $\xi \in \mc P \cup \mc U$ and a branch $\wp \in \mc B$ at $\xi$, choose an inclusion $\bar j_{\xi,\wp}:F_\xi^\sep \hookrightarrow F_\wp^\sep$ extending the given inclusions $j_{\xi,\wp}:F_\xi \subset F_\wp$, and inducing homomorphisms $\alpha_{\wp,\xi}:\Gal(F_\wp^\sep/F_\wp)  =: \Gal(F_\wp) \to \Gal(F_\xi^\sep/F_\xi) =: \Gal(F_\xi)$.  Choose a maximal tree $\mc T$ in the reduction graph $\Gamma$; thus the profinite completion $\wh\pi_1(\Gamma,\mc T)$ is the free 
profinite group with generators $e_\wp$ indexed by the edges of $\Gamma$ that do not lie in $\mc T$.  
Let $e_\wp = 1 \in \wh\pi_1(\Gamma,\mc T)$ for each $\wp \in \mc B$ that is an edge of $\mc T$.
Then the absolute Galois group of $F$ is isomorphic to the quotient of the free product
$\Asterisk_{\xi \in \mc P \cup \mc U} \Gal(F_\xi) \ast \wh\pi_1(\Gamma,\mc T)$ by the relations  
$\alpha_{\wp,U}(g)=e_\wp \alpha_{\wp,P}(g) e_\wp^{-1}$ for all triples $P,U,\wp$ where $\wp \in \mc B$ is a branch at $P \in \mc P$ on $U \in \mc U$, and all $g \in \Gal(F_\wp)$.
\end{prop}

Note that the situation considered in \cite{Stix} involved a connected simplicial complex of dimension at most two with associated groups and group homomorphisms associated to boundary maps.  This abstract situation is applied there to categories with descent data.  But while some descent categories involve self-intersections (e.g.\ $U \times_X U \to X$ for the \'etale topology on $X$), patching provides a descent context without such self-intersections.  As a result, the abstract framework described in \cite{Stix} simplifies in our situation, and it suffices to consider one-dimensional simplicial complexes, viz.~graphs, as above.  

Note also that if $\Gamma$ is a tree, and if compatible separable closures $F^\sep, F_\xi^\sep, F_\wp^\sep$ are chosen as in 
Theorem~\ref{dir lim patches thm}, then the above description of $\Gal(F)$ simplifies to the description given there.

It would be interesting to show that Proposition~\ref{free product quotient} can be deduced from 
Theorem~\ref{general vK}.

\bigskip

\noindent{\bf Author Information:}\\

\noindent David Harbater\\
Department of Mathematics, University of Pennsylvania, Philadelphia, PA 19104-6395, USA\\
email: harbater@math.upenn.edu

\medskip

\noindent Julia Hartmann\\
Department of Mathematics, University of Pennsylvania, Philadelphia, PA 19104-6395, USA\\
email: hartmann@math.upenn.edu

\medskip

\noindent Daniel Krashen\\
Department of Mathematics, University of Georgia, Athens, GA 30602, USA\\
email: dkrashen@math.uga.edu

\medskip

\noindent R.~Parimala\\
Department of Mathematics and Computer Science, Emory University, Atlanta, GA 30322, USA\\
email: parimala@mathcs.emory.edu

\medskip

\noindent V.~Suresh\\
Department of Mathematics and Computer Science, Emory University, Atlanta, GA 30322, USA\\
email: suresh@mathcs.emory.edu

\medskip

\noindent The authors were supported on NSF collaborative FRG grant: DMS-1463733 (DH and JH), DMS-1463901 (DK), DMS-1463882 (RP and VS).  Additional support was provided by NSF collaborative FRG grant DMS-1265290 (DH); NSF RTG grant DMS-1344994 (DK); NSF DMS-1401319 (RP); NSF DMS-1301785 (VS); and a Simons Fellowship (JH).

\end{document}